\documentclass [10pt,twoside]{article}
\usepackage[T1]{fontenc}
\usepackage{multicol}
\usepackage{amsfonts,amsmath,amssymb,amsthm,bm,stmaryrd}
\usepackage[english]{babel}
\usepackage{fancybox}
\usepackage{graphicx}
\usepackage[utf8]{inputenc}
\usepackage{enumerate}
\usepackage{fullpage}
\usepackage{mathrsfs}
\usepackage{dsfont}
\usepackage{titlesec} 
\usepackage{mathdots}
\usepackage{relsize,exscale}
\usepackage{multirow}
\usepackage{float}

\usepackage{color}
\definecolor{purple}{rgb}{0.13,0.54,0.13}
\definecolor{marron}{rgb}{0.64,0.16,0.16}
\definecolor{forestgreen}{rgb}{0.13,0.50,0.13}

\usepackage{pgfplots}
\pgfplotsset{compat=newest}
\usetikzlibrary{shapes}
\usetikzlibrary{plotmarks}

\usepackage{geometry}
\usepackage[french]{layout}
\usepackage{fancyhdr} 
\newcommand{\mb}[1]{\mathbb{#1}}
\newcommand{\mc}[1]{\mathcal{#1}}

\newcommand{\mr}[1]{\mathrm{#1}}
\newcommand{\mbf}[1]{\mathbf{#1}}

\newcommand{\im}{\mathrm{Im}}

\theoremstyle{definition} 

\newtheoremstyle%
   {remarks}
   {1ex}
   {1ex}
   {\itshape}
   {-25pt}
   {}   
   {.}
   {\newline}
   {}

\theoremstyle{remark}
\newtheorem{rem}{\textsf{Remark}}

\theoremstyle{remark}

\theoremstyle{plain}
\newtheorem{theo}{\textsc{Theorem}}[section]
\newtheorem{Lemme}{\textsf{Lemma}}[section]
\newtheorem{corol}{\textsf{Corollary}}[section]
\newtheorem{prop}{\textsc{Proposition}}[section]

\numberwithin{equation}{section}

\usepackage[normalem]{ulem}

\author{  Farid AMMAR KHODJA\footnotemark[1] \and Franz CHOULY\footnotemark[1]  
\and Michel DUPREZ\footnotemark[1] \footnotemark[2]}

\title{Partial null controllability of parabolic linear systems}

\begin{document}
\footnotetext[1]{
Laboratoire de Mathématiques de Besançon UMR CNRS 6623, Université de Franche-Comté,
16 route de Gray, 25030 Besançon Cedex, France}
\footnotetext[2]{Corresponding author : \textit{Email}: michel.duprez@univ-fcomte.fr} 
\maketitle
\renewcommand{\abstractname}{Abstract}

\begin{abstract}
This paper is devoted  to the  partial 
 null controllability issue of parabolic linear systems with $n$ equations. 
 Given a bounded  domain $\Omega$ in $\mb{R}^N$ ($N\in \mb{N}^*$), 
 we study the effect of $m$ localized controls in 
 a nonempty open subset $\omega$ only controlling  $p$ components of the solution ($p,m\leqslant n$). 
 The first main result of this paper is a necessary and sufficient condition 
 when the coupling and control matrices are  constant. 
 The second result provides, in a first step, a sufficient condition of partial null controllability 
 when the matrices only depend on time.  
In a second step, through an example of partially controlled $2\times2$ parabolic system, 
we will provide positive and negative results on partial null controllability 
when the coefficients are space dependent.

\end{abstract}


\section{Introduction and main results}

\hspace*{4mm} Let $\Omega$ be a bounded domain in $\mb{R}^N$ ($N\in \mb{N}^*$) 
with a $\mc{C}^2$-class boundary $\partial \Omega$, 
  $\omega$ be a nonempty open subset of $\Omega$ and $T>0$.  
Let $p$, $m$, $n\in \mb{N}^*$ such that $p,m\leqslant n$. 
We consider in this paper the following  system of $n$ parabolic linear equations
\begin{equation}\label{partiel:syst para lin equa contr intro}
\begin{array}{l}
  \left\{\begin{array}{ll}
   \partial_t y=\Delta y+Ay+B\mathds{1}_{\omega}u&\mr{~in~}Q_T:=\Omega\times (0,T),\\
        y=0&\mr{~on}~\Sigma_T:=\partial \Omega\times (0,T),\\
	y(0)=y_0 &~\mr{in}~\Omega,
        \end{array}
\right.
\end{array}
\end{equation}
where $y_0\in L^2(\Omega)^n$ is the initial data, $u\in L^2(Q_T)^m$ is the control 
and 
 \begin{equation*}
\begin{array}{l}
 A\in L^{\infty}(Q_T;\mc{L}(\mb{R}^n))
 ~\mr{and }~ 
 B\in L^{\infty}(Q_T;\mc{L}(\mb{R}^m,\mb{R}^n)).
\end{array}
 \end{equation*}

In many fields such as chemistry, physics or biology it appeared relevant 
to study the controllability of such a system  
(see \cite{ammar2011recent}). 
For example, in \cite{chakrabarty2009distributed}, the authors study a system of three semilinear heat equations 
which is a model coming from a mathematical description 
of the growth of brain tumors. 
The unknowns  are 
the drug concentration, the density of tumors cells and the density of wealthy cells 
and the aim is  to control only  two  of them with one control. 
This practical issue motivates the introduction of the partial null controllability.



For an initial condition $y(0)=y_0\in L^2(\Omega)^n$ and a control 
$u\in L^2(Q_T)^m$, 
it is well-known that System (\ref{partiel:syst para lin equa contr intro}) admits a unique solution 
in 
$W(0,T)^n$, where  
\begin{equation*}
W(0,T):=\{y\in L^2(0,T;H^1_0(\Omega )),
\partial_ty\in L^2(0,T;H^{-1}(\Omega ))\},
\end{equation*}
with $H^{-1}(\Omega):=H^1_0(\Omega)'$ 
and the following estimate holds (see \cite{lions1968problemes})
\begin{equation}\label{partiel:estim sol int intro}
 \|y\|_{L^2(0,T;H_0^1(\Omega)^n)}
+\|y\|_{\mc{C}^0([0,T];L^2(\Omega)^n)}
\leqslant C(\|y_0\|_{L^2(\Omega)^n}
+\|u\|_{L^2(Q_T)^m}),
\end{equation}
where $C$ does not depend on time.
We denote by $y(\cdot ;y_0,u)$ the solution to System (\ref{partiel:syst para lin equa contr intro}) 
determined by the couple $(y_0,u)$.

Let us consider $\Pi_p$ the \textit{projection matrix} of $\mc{L}(\mb{R}^n)$ given by $\Pi_p:=(I_p~~ 0_{p,n-p})$ 
($I_p$ is the identity matrix of $\mc{L}(\mb{R}^p)$ and $0_{p,n-p}$ the null matrix of $\mc{L}(\mb{R}^{n-p},\mb{R}^p)$), 
that is,  
\begin{equation*}\begin{array}{cccc}
 \Pi_p:&\mb{R}^n&\rightarrow &\mb{R}^p, \\
&(y_1,...,y_n)&\mapsto&(y_1,...,y_p).
\end{array} \end{equation*}
System  (\ref{partiel:syst para lin equa contr intro}) is said to be
\begin{itemize}
\item 
$\mbf{\Pi_p}$-\textbf{approximately controllable} 
on the time interval $(0,T)$,  if 
for all real number $\varepsilon>0$ and $y_0,~y_T\in L^2(\Omega)^n$ 
there exists a control $u\in L^2(Q_T)^m$ 
such that 
\begin{equation*}\begin{array}{c}
\|\Pi_py(T;y_0,u)-\Pi_py_T\|_{L^2(\Omega)^p}
\leqslant \varepsilon.
\end{array}\end{equation*}
\item 
$\mbf{\Pi_p}$-\textbf{null controllable} 
on the time interval $(0,T)$,  if 
for all initial condition $y_0\in L^2(\Omega)^n$, 
there exists a control  $u\in L^2(Q_T)^m$
such that 
\begin{equation*}\begin{array}{c}
\Pi_p y(T;y_0,u)\equiv0 \mr{~in~} \Omega.
\end{array}\end{equation*}
 \end{itemize}

%


Before stating our main results, let us recall 
 the few known results about the (full) null controllability
 of System (\ref{partiel:syst para lin equa contr intro}). 
The first of them is about 
cascade systems (see \cite{gonzalez2010controllability}). 
The authors prove the null controllability 
of System (\ref{partiel:syst para lin equa contr intro}) with the control matrix 
$B:=e_1$ (the first vector of the canonical basis of $\mb{R}^n$) 
and a coupling matrix $A$ of the form
\begin{equation}\label{partiel:form cascade}
 A:=\left(
 \begin{array}{ccccc}
  \alpha_{1,1}&\alpha_{1,2}&\alpha_{1,3}&\cdots&\alpha_{1,n}\\
  \alpha_{2,1}&\alpha_{2,2}&\alpha_{2,3}&\cdots&\alpha_{2,n}\\
 0&\alpha_{3,2}&\alpha_{3,3}&\cdots&\alpha_{3,n}\\
 \vdots&\vdots&\ddots&\ddots&\vdots\\
 0&0&\cdots &\alpha_{n,n-1}&\alpha_{n,n}
 \end{array}
\right),
\end{equation}
where the coefficients $\alpha_{i,j}$ are elements of $L^{\infty}(Q_T)$ 
for all $i,j\in\{1,...,n\}$ 
and satisfy for a  positive constant $C$ and a nonempty open set $\omega_0$ of $\omega$
\begin{equation*}
\alpha_{i+1,i}\geqslant C \mr{~in~}\omega_0
\mr{~~~or~~~}-\alpha_{i+1,i}\geqslant C \mr{~in~}\omega_0
\mr{~~~for~all~}i\in \{1,...,n-1\}.
\end{equation*}
A similar result on parabolic systems with cascade coupling matrices can be found in \cite{alabau2013}. 

The null controllability of parabolic $3\times3$ linear systems with  space/time 
dependent coefficients and non cascade structure
is  studied in  \cite{benabdallah2014} and \cite{mauffrey2013null} 
(see also  \cite{gonzalez2010controllability}).

If $A\in \mc{L}(\mb{R}^n)$ and $B\in \mc{L}(\mb{R}^m,\mb{R}^n)$ (the constant case), 
it has been proved in \cite{ammar2009kalman} that System (\ref{partiel:syst para lin equa contr intro}) is 
null controllable on the time interval $(0,T)$ 
if and only if the following condition holds
  \begin{equation}\label{partiel:kalman considion constant}
   \mr{rank}~[A|B]=n,
  \end{equation}
where $[A|B]$, the so-called \textit{Kalman matrix}, is defined as
\begin{equation}\label{partiel:def matrice kalman}
 [A|B]:=(B|AB|...|A^{n-1}B).
\end{equation}

 For time dependent coupling and control matrices, 
we need  some additional regularity. More precisely,  
we need to suppose that $A\in \mc{C}^{n-1}([0,T];\mc{L}(\mb{R}^n))$ 
and $ B\in\mc{C}^{n}([0,T];\mc{L}(\mb{R}^m;\mb{R}^n))$. 
In this case, the associated Kalman matrix is defined as follows. 
Let us define
\begin{equation*}
 \left\{\begin{array}{l}
       B_0(t):=B(t),\\ 
       B_i(t):=A(t)B_{i-1}(t)-\partial_tB_{i-1}(t)\mr{~~~for~all}~ i\in\{1,..., n-1\}
        \end{array}
\right.
\end{equation*}
and denote by $[A | B](\cdot) 
\in \mc{C}^1 ([0, T ]; \mc{L}(\mb{R}^{nm} ; \mb{R}^n ))$ 
the matrix function given by 
\begin{equation}\label{partiel:def B_i}
 [A | B](\cdot) :=(B_0(\cdot)|B_1(\cdot)|...|B_{n-1}(\cdot)).
\end{equation}
In \cite{ammar2009generalization} the authors prove first that,  
if there exists $t_0\in [0,T]$ 
such that
  \begin{equation}\label{partiel:condition theoreme temps null1}
   \mr{rank}~[A | B](t_0)=n,
  \end{equation}
  then System (\ref{partiel:syst para lin equa contr intro}) is  
 null controllable on the time interval $(0,T)$. 
Secondly that System (\ref{partiel:syst para lin equa contr intro}) 
is  null controllable 
on every interval $(T_0,T_1)$ with
$ 0\leqslant T_0 < T_1\leqslant T$  if and only if there exists 
a dense subset $E$ of $(0, T )$ such that
\begin{equation}\label{partiel:condition theoreme temps null2}
\mr{rank}~ [A | B](t) = n  \mr{~for~ every}~t \in E.
\end{equation}

In the present paper, the controls are acting on several equations 
but on one subset $\omega$ of $\Omega$. 
Concerning the case where the control domains are not identical, 
we refer to  \cite{olive2012}.


Our first result is the following:
\begin{theo}\label{partiel:theorem contro coupl constant}
Assume that the coupling and control matrices are constant in space and time, i. e., 
$A\in \mc{L}(\mb{R}^n)$ and $B\in \mc{L}(\mb{R}^m,\mb{R}^n)$. 
The condition 
 \begin{equation}\label{partiel:cond cas constant une force}
   \mr{rank}~\Pi_p[A|B]=p
  \end{equation}
  is equivalent to the $\Pi_p$-null/approximate controllability on the time interval $(0,T)$ of 
System (\ref{partiel:syst para lin equa contr intro}).
\end{theo}

The Condition (\ref{partiel:cond cas constant une force}) 
for $\Pi_p$-null controllability reduces to  Condition 
(\ref{partiel:kalman considion constant}) whenever $p=n$. A second result concerns  the non-autonomous case:

\begin{theo}\label{partiel:theo temps}
Assume that $A\in \mc{C}^{n-1}([0,T];\mc{L}(\mb{R}^n))$ 
and $ B\in\mc{C}^{n}([0,T];\mc{L}(\mb{R}^m;\mb{R}^n))$. 
  If 
  \begin{equation}\label{partiel:condition theoreme temps}
   \mr{rank}~\Pi_p[A|B](T)=p,
  \end{equation}
  then System (\ref{partiel:syst para lin equa contr intro}) is  
  $\Pi_p$-null/approximately controllable on the time interval $(0,T)$.
\end{theo}

 In Theorems \ref{partiel:theorem contro coupl constant} and \ref{partiel:theo temps}, we control the $p$ first components of the solution $y$. 
 If we want to control some other components a permutation of lines leads to the same situation.

\begin{rem}
\begin{enumerate}
 \item When the components of the matrices $A$ and $B$ are analytic functions on the time interval $[0,T]$, 
 Condition \eqref{partiel:condition theoreme temps null1} is necessary for the  null controllability 
 of System (\ref{partiel:syst para lin equa contr intro}) (see Th. 1.3 in \cite{ammar2009generalization}). 
 Under the same assumption, the proof of this result can be adapted to show that the following condition
   \begin{equation*}
   \left\{\begin{array}{l}
  \mr{there~exists~}t_0\in [0,T]\mr{~such~that:~}\\
  \mr{rank}~\Pi_p[A | B](t_0)=p,           
          \end{array}
\right.
  \end{equation*}
 is necessary to the $\Pi_p$-null controllability of System (\ref{partiel:syst para lin equa contr intro}).
 \item As told before, under Condition (\ref{partiel:condition theoreme temps null1}), 
System (\ref{partiel:syst para lin equa contr intro}) is   null controllable. 
  But unlike the case where all the components are controlled, 
  the $\Pi_p$-null controllability at a time $t_0$ smaller than $T$ 
  does not imply this property on the time interval $(0,T)$. 
  This roughly explains Condition (\ref{partiel:condition theoreme temps}). 
 Furthermore this condition  can not be necessary under the assumptions of Theorem \ref{partiel:theo temps} 
  (for a counterexample we refer to \cite{ammar2009generalization}). 
  
\end{enumerate}
%

  \end{rem}

\begin{rem}
In the proofs of Theorems \ref{partiel:theorem contro coupl constant} and \ref{partiel:theo temps}, we will use 
a result of  
null controllability for cascade systems 
(see Section \ref{partiel:section useful}) proved in  \cite{ammar2009generalization,gonzalez2010controllability}  
where the authors consider 
a time-dependent second order elliptic operator $L(t)$ 
 given by
\begin{equation}\label{partiel:obs operator carleman}
 L(t)y(x,t)=-\displaystyle\sum_{i,j=1}^N
 \dfrac{\partial}{\partial x_i}\left(\alpha_{i,j}(x,t)\dfrac{\partial y}{\partial x_j}(x,t)\right)
 +\displaystyle\sum_{i=1}^Nb_i(x,t) \dfrac{\partial y}{\partial x_i}(x,t)
+c(x,t)y(x,t),
\end{equation}
with coefficients $\alpha_{i,j}$, $b_i$, $c$ satisfying
\begin{equation*}
 \left\{\begin{array}{l}
         \alpha_{i,j}\in W^{1}_{\infty}(Q_T),~b_i,c\in L^{\infty}(Q_T)~1\leqslant i,j\leqslant N,\\
         \alpha_{i,j}(x,t)=\alpha_{j,i}(x,t)~\forall (x,t)\in Q_T,~1\leqslant i,j\leqslant N
        \end{array}
 \right.
\end{equation*}
and the uniform elliptic condition: there exists $a_0>0$ such that
\begin{equation*}
 \sum\limits_{i,j=1}^N\alpha_{i,j}(x,t)\xi_i\xi_j\geqslant a_0|\xi|^2,~\forall (x,t)\in Q_T.
\end{equation*}
 Theorems \ref{partiel:theorem contro coupl constant} and \ref{partiel:theo temps} remain true 
if we replace 
$-\Delta$  by an operator $L(t)$ 
in System \eqref{partiel:syst para lin equa contr intro}.
\end{rem}

   Now the following question arises: what happens in the case of space and time dependent 
   coefficients  ? 
   As it will be shown in the following 
    example, the answer seems to be much more tricky. 
 Let us now consider the following parabolic system of two equations
 \begin{equation}\label{partiel:intro syst prim}
\begin{array}{l}
  \left\{\begin{array}{ll}
   \partial_t y=\Delta y+\alpha z+\mathds{1}_{\omega}u&\mr{~in~}Q_T,\\
    \partial_t z=\Delta z&\mr{~in~}Q_T,\\
        y=z=0&\mr{~on}~\Sigma_T,\\
	y(0)=y_0,~z(0)=z_0 &~\mr{in}~\Omega,
        \end{array}
\right.
\end{array}
\end{equation}
for given initial data $y_0,~z_0\in L^2(\Omega)$, a control $u\in L^2(Q_T)$ and 
where the coefficient $\alpha\in L^{\infty}(\Omega)$.


%

 \begin{theo}\label{partiel:theo example contr intro}
\begin{enumerate}[(1)]
\item Assume that  $\alpha\in C^1([0,T])$. 
Then System  (\ref{partiel:intro syst prim}) is $\Pi_1$-null controllable for any open set 
$\omega\subset \Omega\subset \mb{R}^N$ $(N\in\mb{N}^*)$, 
that is for all initial conditions $y_0,z_0\in L^2(\Omega)$, 
there exists a control $u\in L^2(Q_T)$ such that 
the solution $(y,z)$ to System (\ref{partiel:intro syst prim}) satisfies $y(T)\equiv 0$ in $\Omega$.
\item Let $\Omega:=(a,b)\subset\mb{R}$ $(a,b\in \mb{R})$, $\alpha\in L^{\infty}(\Omega)$, $(w_k)_{k\geqslant1}$ be the $L^2$-normalized eigenfunctions 
 of $-\Delta$ in $\Omega$  with Dirichlet boundary conditions and for all $k,l\in\mb{N}^*$, 
  \begin{equation*}\label{partiel:def alpha kj}
\alpha_{kl}:=\displaystyle\int_{\Omega}\alpha(x) w_k(x)w_l(x)\,dx  .
 \end{equation*}
If the function $\alpha$ satisfies
 \begin{equation}\label{partiel:second point t space}
 \left| \alpha_{kl}\right|
 \leqslant C_{1}e^{-C_{2}|k-l|}\mr{~for~ all }~k,l\in\mb{N}^*,
 \end{equation}
for two positive constants $C_{1}>0$ and  $C_{2}>b-a$, 
 then System (\ref{partiel:intro syst prim}) is $\Pi_1$-null controllable for any open set $\omega\subset \Omega$.
 \item Assume that   $\Omega:=(0,2\pi)$ and $\omega\subset (\pi,2\pi)$. 
Let us consider $\alpha\in L^{\infty}(0,2\pi)$ defined by
\begin{equation*}
\alpha(x):=\sum\limits_{j=1}^{\infty}\dfrac{1}{j^2}\cos(15jx)\mr{~for~all~}x\in(0,2\pi).
\end{equation*}
Then 
System (\ref{partiel:intro syst prim}) is not $\Pi_1$-null controllable. 
More precisely, there exists $k_1\in\{1,...,7\}$ 
 such that for the initial condition $(y_0,z_0)=(0,\sin(k_1x))$ and  any control 
 $u\in L^2(Q_T)$ the solution $y$ 
 to System  (\ref{partiel:intro syst prim}) is not identically equal to zero at time $T$.
%
%
\end{enumerate}
\end{theo}
We will not prove item (1) in  Theorem \ref{partiel:theo example contr intro},  because it is a direct consequence of Theorem \ref{partiel:theo temps}.

\begin{rem}
Suppose that $\Omega:=(0,\pi)$. 
 Consider  $\alpha\in L^{\infty}(0,\pi)$ and the real sequence $(\alpha_p)_{p\in \mb{N}}$ such that  
 for all $x\in (0,\pi)$ 
\begin{equation*}
\alpha(x):=\sum\limits_{p=0}^{\infty}\alpha_p\cos(px). 
\end{equation*}
Concerning  item (2), we remark that  Condition (\ref{partiel:second point t space}) is equivalent to  the existence of 
two  constants $C_{1}>0,~C_{2}>\pi$ such that, for all $p\in\mb{N}$,
 \begin{equation*}
 | \alpha_{p}|\leqslant C_{1}e^{-C_{2}p}.
 \end{equation*}
As it will be shown, the proof of item (3) in Theorem \ref{partiel:theo example contr intro} 
can be adapted in order to get the same conclusion 
for any $\alpha\in H^k(0,2\pi)$ ($k\in \mb{N}^*$) defined by
 \begin{equation}\label{partiel:form alpha gene}
\alpha(x):=\sum\limits_{j=1}^{\infty}\dfrac{1}{j^{k+1}}\cos((2k+13)jx)
\mr{~for~all~}x\in(0,2\pi).
\end{equation}
These given functions $\alpha$ belong to $H^k(0,\pi)$ but not to $D((-\Delta)^{k/2})$. 
Indeed,  in the proof of the third item in Theorem  \ref{partiel:theo example contr intro}, 
we use the fact that the matrix $(\alpha_{kl})_{k,l\in \mb{N}^*}$ is sparse (see \eqref{partiel:calcul submat alpha}), 
what seems true only  for coupling terms $\alpha$ of the form \eqref{partiel:form alpha gene}. 
Thus $\alpha$ is not zero on the boundary.
\end{rem}

\begin{rem}
From Theorem \ref{partiel:theo example contr intro}, one can deduce some new results concerning the null controllability of 
the heat equation with a right-hand side. Consider the system
  \begin{equation}\label{partiel:intro heat second member}
\begin{array}{l}
  \left\{\begin{array}{ll}
   \partial_t y=\Delta y+f+\mathds{1}_{\omega}u&\mr{~in~}(0,\pi)\times(0,T),\\
        y(0)=y(\pi)=0&\mr{~on}~(0,T),\\
	y(0)=y_0 &~\mr{in}~(0,\pi),
        \end{array}
\right.
\end{array}
\end{equation}
where $y_0\in L^2(0,\pi)$ is the initial data and $f,u\in L^2(Q_T)$ are the right-hand side and the control, respectively.  
Using the Carleman inequality (see \cite{fursikov1996controllability}), one can prove that 
 System \eqref{partiel:intro heat second member} is null controllable when $f$ satisfies 
\begin{equation}\label{partiel:cond f heat sec mem}
e^{\frac{C}{T-t}}f\in L^2(Q_T), 
\end{equation}
 for a positive constant $C$. 
 For more general right-hand sides it was rather open. 
 The second and third points of Theorem \ref{partiel:theo example contr intro} provide some positive and negative null controllability results 
 for System \eqref{partiel:intro heat second member} with right-hand side $f$ which does not  fulfil Condition  \eqref{partiel:cond f heat sec mem}.
\end{rem}

\begin{rem} Consider the same system as System \eqref{partiel:intro syst prim} except that the control is now on the boundary, 
that is
\begin{equation}\label{partiel:intro syst bord}
\begin{array}{l}
  \left\{\begin{array}{ll}
   \partial_t y=\Delta y+\alpha z&\mr{~in~}(0,\pi)\times(0,T),\\
    \partial_t z=\Delta z&\mr{~in~}(0,\pi)\times(0,T),\\
        y(0,t)=v(t), ~y(\pi,t)=z(0,t)=z(\pi,t)=0&\mr{~on}~(0,T),\\
	y(x,0)=y_0(x),~z(x,0)=z_0(x) &~\mr{in}~(0,\pi),
        \end{array}
\right.
\end{array}
\end{equation}
where $y_0,~z_0\in H^{-1}(0,\pi)$. 
In Theorem \ref{partiel:theo contr exemple bord}, we provide an explicit  coupling function $\alpha$ for which the $\Pi_1$-null controllability of System \eqref{partiel:intro syst bord} does not hold. 
Moreover one can adapt the proof of the second point in Theorem \ref{partiel:theo example contr intro} to prove the $\Pi_1$-null controllability of System \eqref{partiel:intro syst bord} 
under Condition \eqref{partiel:second point t space}.
\end{rem}

If the coupling matrix depends on space, the notions of $\Pi_1$-null and approximate 
controllability are not 
necessarily equivalent. Indeed, 
according to the choice of the coupling function  $\alpha\in L^{\infty}(\Omega)$, 
System (\ref{partiel:intro syst prim}) can be $\Pi_1$-null controllable or not. 
But this system is $\Pi_1$-approximately controllable for all $\alpha\in L^{\infty}(\Omega)$:
 \begin{theo}
Let $\alpha\in L^{\infty}(Q_T)$.
Then System (\ref{partiel:intro syst prim}) is $\Pi_1$-approximately controllable 
for any open set $\omega\subset \Omega\subset \mb{R}^N$ $(N\in\mb{N}^*)$, 
that is for all $y_0,y_T,z_0\in L^2(\Omega)$ and all $\varepsilon>0$, 
there exists a control $u\in L^2(Q_T)$ such that 
the solution $(y,z)$ to System (\ref{partiel:intro syst prim}) satisfies 
$$\|y(T)-y_T\|_{L^2(\Omega)}\leqslant\varepsilon.$$
\end{theo}
\noindent This result is a direct consequence of 
the unique continuation property and existence/unicity of solutions for a single heat equation. 
Indeed System  (\ref{partiel:intro syst prim}) is $\Pi_1$-approximately controllable  
(see Proposition \ref{partiel:property equivalence contr obs}) 
if and only if for all $\phi_0\in L^2(\Omega)$ the solution to the adjoint system 
  \begin{equation}\label{partiel:syst dual approx}
\begin{array}{l}
  \left\{\begin{array}{ll}
-\partial_t\phi=\Delta \phi&\mr{~in~}Q_T,\\
-\partial_t\psi=\Delta \psi+\alpha\phi&\mr{~in~}Q_T,\\
        \phi=\psi=0&\mr{~on}~\Sigma_T,\\
	\phi(T)=\phi_0,~\psi(T)=0 &~\mr{in}~\Omega
        \end{array}
\right.
\end{array}
\end{equation}
satisfies
\begin{equation*}
 \phi\equiv0\mr{~in~} \omega\times(0,T)~\Rightarrow  (\phi,\psi)\equiv0\mr{~in~}Q_T.
\end{equation*}
If we assume that, for an initial data $\phi_0\in L^2(\Omega)$, 
the solution to System \eqref{partiel:syst dual approx} 
satisfies $\phi\equiv0$ in $\omega\times(0,T) $, 
then using  Mizohata uniqueness Theorem in  
\cite{mizohata58}, $\phi\equiv0$ in $Q_T$ and consequently $\psi\equiv0$ in $Q_T$.
For another example of parabolic systems 
for which these notions are not equivalent we refer for instance to \cite{CherifMinimalTimeDisjoint}. 

\begin{rem}
The quantity $\alpha_{kl}$, which appears in the second item of Theorem \ref{partiel:theo example contr intro}, has already been considered 
in some controllability studies for parabolic systems. Let us define for all $k\in \mb{N}^*$
\begin{equation*}
\left\{\begin{array}{l}
 I_{1,k}(\alpha):=\displaystyle\int_0^{a}
 \textstyle\alpha(x)w_k(x)^2dx,\vspace*{3mm}\\
 I_k(\alpha):=\alpha_{kk}.
\end{array}\right.
\end{equation*}
In \cite{minimaltime}, the authors have proved that the system
\begin{equation}\label{partiel:intro syst cascade}
\begin{array}{l}
  \left\{\begin{array}{ll}
   \partial_t y=\Delta y+\alpha z&\mr{~in~}(0,\pi)\times(0,T),\\
    \partial_t z=\Delta z+\mathds{1}_{\omega}u&\mr{~in~}(0,\pi)\times(0,T),\\
        y(0,t)=y(\pi,t)=z(0,t)=z(\pi,t)=0&\mr{~on}~(0,T),\\
	y(x,0)=y_0(x),~z(x,0)=z_0(x) &~\mr{in}~(0,\pi),
        \end{array}
\right.
\end{array}
\end{equation}
is  approximately controllable if and only if 
\begin{equation*}
\begin{array}{c}
  |I_k(\alpha)|+|I_{1,k}(\alpha)|\neq0\mr{~for~all~}k\in\mb{N}^*.     
\end{array}
\end{equation*}
A similar result has been obtained for the boundary  approximate controllability in \cite{oliveapprox2014}. 
Consider now
\begin{equation*}
 T_0(\alpha):=\limsup\limits_{\substack{k \to \infty }} 
 \frac{-\log (\min\{\left| I_k \right|,\left| I_{1,k} \right|\}) }{k^2} .
\end{equation*}
It is also proved in \cite{minimaltime} that: 
If $T>T_0(\alpha)$, then System \eqref{partiel:intro syst cascade} is  null controllable at time $T$ 
and if $T<T_0(\alpha)$, then System \eqref{partiel:intro syst cascade} is not  null controllable at time $T$. 
As in the present paper, we observe a difference between the approximate and null controllability, in contrast with the scalar case (see \cite{ammar2011recent}).
\end{rem}

In this paper, the sections are organized as follows. 
We start with some preliminary results on the null controllability for the  cascade systems 
and on the dual concept   associated to the $\Pi_p$-null controllability. 
Theorem \ref{partiel:theorem contro coupl constant}
  is proved in a first step with one force 
  i.e. $B\in \mb{R}^n$ in Section \ref{partiel:section one force} 
  and in a second step with $m$ forces in Section \ref{partiel:section m forces}. 
  Section \ref{partiel:section temps} is devoted to proving  Theorem \ref{partiel:theo temps}. 
  We consider the situations of the second and  third items of  
  Theorem \ref{partiel:theo example contr intro} in Section \ref{partiel:example control} 
  and  \ref{partiel:section contre exemple} respectively.   
  This paper ends with some numerical illustrations of $\Pi_1$-null controllability 
  and non $\Pi_1$-null controllability of System \eqref{partiel:intro syst prim} in 
Section  \ref{partiel:section num}.

%
%
%
\section{Preliminaries}\label{partiel:section useful}

\hspace*{4mm} In this section, we recall a known result about cascade systems 
and provide a characterization of the $\Pi_p$-controllability 
through the corresponding adjoint system.

\subsection{Cascade systems}

\hspace*{4mm} Some theorems of this paper use the following result of null controllability for the 
following cascade system of $n$ equations controlled by $r$ distributed functions
\begin{equation}\label{partiel:syst cascade}
\begin{array}{l}
  \left\{\begin{array}{ll}
   \partial_t w=\Delta w+Cw+D\mathds{1}_{\omega}u&\mr{~in~}Q_T,\\
        w=0&\mr{~on}~\Sigma_T,\\
	w(0)=w_0 &~\mr{in}~\Omega,
        \end{array}
\right.
\end{array}
\end{equation}
where $w_0\in L^2(\Omega)^n$, $u=(u_1,...,u_r)\in L^2(Q_T)^r$, 
with $r\in \{ 1,...,n\}$, 
and the coupling and control matrices $C\in\mc{C}^0([0,T]; \mc{L}(\mb{R}^n))$ and 
$D\in\mc{L}(\mb{R}^r,\mb{R}^n)$  
are given by
\begin{equation}\label{partiel:matrice cascade}
 C(t):=\left(
 \begin{array}{cccc}
  C_{11}(t)&C_{12}(t)&\cdots&C_{1r}(t)\\
  0&C_{22}(t)&\cdots&C_{2r}(t)\\
 \vdots&\vdots&\ddots&\vdots\\
 0&0&\cdots &C_{rr}(t)
 \end{array}
\right)
\end{equation}
with
\begin{equation*}
 C_{ii}(t):=\left(
 \begin{array}{ccccc}
  \alpha_{11}^i(t)&\alpha^i_{12}(t)&\alpha^i_{13}(t)&\cdots&\alpha^i_{1,s_i}(t)\\
  1&\alpha^i_{22}(t)&\alpha^i_{23}(t)&\cdots&\alpha^i_{2,s_i}(t)\\
 0&1&\alpha^i_{33}(t)&\cdots&\alpha^i_{3,s_i}(t)\\
 \vdots&\vdots&\ddots&\ddots&\vdots\\
 0&0&\cdots &1&\alpha^i_{s_i,s_i}(t)
 \end{array}
\right),
\end{equation*}
$s_i\in \mb{N}$, $\sum_{i=1}^rs_i=n$ and  
 $D:=(e_{S_1}|...|e_{S_r})$ with $S_1=1$ and $S_i=1+\sum_{j=1}^{i-1}s_j$, 
$i\in \{ 2,...,r\}$ 
 ($e_j$ is the $j$-th element of the canonical basis of $\mb{R}^n$).

\begin{theo}\label{partiel:theo cascade}
System (\ref{partiel:syst cascade}) is null controllable on the time interval $(0,T)$, i.e. 
 for all $w_0\in L^2(\Omega)^n$ there exists $u\in L^2(\Omega)^r$ 
 such that the solution $w$ in $W(0,T)^n$ to System  (\ref{partiel:syst cascade}) 
 satisfies $w(T)\equiv0$ in $\Omega$.
\end{theo}
The proof of this result uses a Carleman estimate (see \cite{fursikov1996controllability}) and can be found in \cite{ammar2009generalization} or 
\cite{gonzalez2010controllability}. 

\subsection[Adjoint system]{Partial null controllability 
of a parabolic linear system by $m$ forces 
and adjoint system}\label{partiel:subsec dual}

\hspace*{4mm} It is nowadays well-known that the controllability has a dual concept 
called \textit{observability} 
(see for instance \cite{ammar2011recent}). 
We detail below the observability for the $\Pi_p$-controllability.

\begin{prop}\label{partiel:property equivalence contr obs}
\begin{enumerate}
\item  
System 
(\ref{partiel:syst para lin equa contr intro})
is $\Pi_p$-null controllable on the time interval $(0,T)$ 
if and only if there exists a constant $C_{obs}>0$ such that
 for all $\varphi_0=(\varphi^0_1,...,\varphi_p^0)\in L^2(\Omega)^p$  
 the solution $\varphi\in W(0,T)^n$ to the adjoint system 
\begin{equation} \label{partiel:obs syst adjoint}
\begin{array}{l}
 \left\{\begin{array}{ll}
 -\partial_t\varphi=\Delta \varphi
 +A^* \varphi&~\mr{in}~Q_T,\\
 \varphi=0&\mr{~on~}\Sigma_T,\\
\varphi(\cdot,T)=\Pi_p^*\varphi_0=(\varphi^0_1,...,\varphi_p^0,0,...,0)&~\mr{in}~\Omega
\end{array}\right.
\end{array}
\end{equation}
satisfies  the observability inequality 
 \begin{equation}\label{partiel:inegalite d observabilite lineraire2}
\|\varphi(0)\|_{L^2(\Omega)^n}^2
  \leqslant  C_{obs} \displaystyle\int_{0}^T\|B^*\varphi\|_{L^2(\omega)^m}^2\,dt.
 \end{equation}
\item System (\ref{partiel:syst para lin equa contr intro}) is $\Pi_p$-approximately 
controllable on the time interval $(0,T)$ 
if and only if for all $\varphi_0\in L^2(\Omega)^p$ 
the solution $\varphi$ to System  (\ref{partiel:obs syst adjoint}) satisfies
\begin{equation*}
 B^*\varphi\equiv0\mr{~in~}(0,T)\times \omega~\Rightarrow  \varphi\equiv0\mr{~in~}Q_T.
\end{equation*}
\end{enumerate}
\end{prop}

\begin{proof}
For all $y_0\in L^2(\Omega)^n$, 
and $u\in L^2(Q_T)^m$, 
we denote by $y(t;y_0,u)$ the solution 
to System (\ref{partiel:syst para lin equa contr intro}) at time $t\in [0,T]$. 
For all $t\in [0,T]$, let us consider the operators $S_t$ and $L_t$  defined as follows
 \begin{equation}\label{partiel:def St Lt}
  \begin{array}{cccc}
   S_t:
  & L^2(\Omega)^n&\rightarrow&L^2(\Omega)^n\\
&y_0&\mapsto& y(t;y_0,0)
 \end{array}
 \mr{~~and~~}
   \begin{array}{cccc}
   L_t:&L^2(Q_T)^m&\rightarrow&L^2(\Omega)^n\\
 &u&\mapsto& y(t;0,u).
 \end{array}
 \end{equation}
\begin{enumerate}
\item 
System (\ref{partiel:syst para lin equa contr intro}) is $\Pi_p$-null controllable on the time interval $(0,T)$ 
if and only if 
 \begin{equation}\label{partiel:LTu=-STu_0}
 \begin{array}{c}
  \forall y_0\in L^2(\Omega)^n,~
  \exists  u\in L^2(Q_T)^m \mr{~such~that}\\
  \Pi_p L_Tu=-\Pi_p S_Ty_0.
 \end{array}
 \end{equation} 
Problem (\ref{partiel:LTu=-STu_0}) admits a solution if and only if 
\begin{equation}\label{partiel:im PST subset im PLT}
 \im ~\Pi_p S_T\subset \im ~\Pi_p L_T.
\end{equation}
The inclusion (\ref{partiel:im PST subset im PLT}) is equivalent to 
(see \cite{coron2009control}, Lemma 2.48 p. 58) 
\begin{equation}\label{partiel:prop ine obs1}
\begin{array}{c}
 \exists C>0\mr{~such~that~}\forall \varphi_0 \in L^2(\Omega)^p,\\
 \|S_T^*\Pi_p ^*\varphi_0\|_{L^2(\Omega)^n}^2
 \leqslant C\|L_T^*\Pi_p ^*\varphi_0\|^2_{L^2(Q_T)^m}.
\end{array}
 \end{equation}
We note that
\begin{equation*}
  \begin{array}{cccc}
   S_T^*\Pi_p ^*:&L^2(\Omega)^p&\rightarrow & L^2(\Omega)^n\\
 &\varphi_0&\mapsto&\varphi(0)
 \end{array}
 \mr{~~and~~}
   \begin{array}{cccc}
   L_T^*\Pi_p ^*:&L^2(\Omega)^p&\rightarrow& L^2(Q_T)^m\\
 &\varphi_0&\mapsto&\mathds{1}_{\omega}B^*\varphi,
 \end{array}
 \end{equation*} 
where 
$\varphi\in W(0,T)^n$ is the solution to System (\ref{partiel:obs syst adjoint}). 
Indeed, for all $y_0\in L^2(\Omega)^n$, 
$u\in L^2(Q_T)^m$ and $\varphi_0\in L^2(\Omega)^p$
\begin{equation}\label{partiel:prop ine obs2}
\begin{array}{rcl}
   \langle \Pi_p S_Ty_0,\varphi_0\rangle_{L^2(\Omega)^p}
  &=&\langle y(T;y_0,0), \varphi(T)\rangle_{L^2(\Omega)^n}\\
  &=&\displaystyle\int_{0}^T\langle \partial_ty(s;y_0,0),\varphi(s) \rangle_{L^2(\Omega)^n}\mr{d}s\\
  &&~~~~~~~+\displaystyle\int_{0}^T\langle y(s;y_0,0),\partial_t\varphi(s) \rangle_{L^2(\Omega)^n}\mr{d}s
  +\langle y_0, \varphi(0)\rangle_{L^2(\Omega)^n}\\
   &=&\langle y_0, \varphi(0)\rangle_{L^2(\Omega)^n}
\end{array}
\end{equation}
and
\begin{equation}\label{partiel:prop ine obs3}
\begin{array}{rcl}
  \langle \Pi_pL_Tu,\varphi_0\rangle_{L^2(\Omega)^p}
  &=&\langle y(T;0,u), \varphi(T)\rangle_{L^2(\Omega)^n}\\
  &=&\displaystyle\int_{0}^T\langle \partial_ty(s;0,u),\varphi(s) \rangle_{L^2(\Omega)^n}\mr{d}s
  +\displaystyle\int_{0}^T\langle y(s;0,u),\partial_t\varphi(s) \rangle_{L^2(\Omega)^n}\mr{d}s\\
   &=&\langle \mathds{1}_{\omega}
   B u,\varphi \rangle_{L^2(Q_T)^n}
   =\langle 
   u,\mathds{1}_{\omega}B^*\varphi \rangle_{L^2(Q_T)^m}.
\end{array}
\end{equation}
The  inequality (\ref{partiel:prop ine obs1}) combined with (\ref{partiel:prop ine obs2})-(\ref{partiel:prop ine obs3}) 
 lead to the conclusion.
\item In view of the definition in (\ref{partiel:def St Lt}) of $S_T$ and $L_T$,  
System (\ref{partiel:syst para lin equa contr intro}) is 
$\Pi_p$-approximately controllable on the time interval $(0,T)$ if and only if 
 \begin{equation*}\begin{array}{c}
\forall (y_0,y_T)\in L^2(\Omega)^n\times L^2(\Omega)^p,~ \forall
\varepsilon>0,~
\exists  u\in L^2(Q_T)^m \mr{~such~ that}\\
  \|\Pi_p L_Tu+\Pi_p S_Ty_0-y_T\|_{L^2(\Omega)^p}\leqslant \varepsilon.
 \end{array}\end{equation*} 
 This is equivalent to
  \begin{equation*}\begin{array}{c}
 \forall\varepsilon>0 ,~ \forall z_T\in L^2(\Omega)^p, 
  \exists  u\in L^2(Q_T)^m \mr{~such~ that}\\
  \|\Pi_p L_Tu-z_T\|_{L^2(\Omega)^p}\leqslant \varepsilon.
  \end{array}\end{equation*}
 That means 
 \begin{equation*}
 \overline{ \Pi_p L_T(L^2(Q_T)^m)}=L^2(\Omega)^p.
 \end{equation*}
In other words
 \begin{equation*}
  \mr{ker}~L_T^*\Pi_p ^*=\{0\}.
 \end{equation*}
 Thus System (\ref{partiel:syst para lin equa contr intro}) is 
 $\Pi_p$-approximately controllable on the time interval $(0,T)$ if and only if 
for all $\varphi_0\in L^2(\Omega)^p$
 \begin{equation*}
 L_T^*\Pi_p ^*\varphi_0=\mathds{1}_{\omega}B^*\varphi\equiv0\mr{~in~}Q_T~\Rightarrow  \varphi\equiv0\mr{~in~}Q_T.
\end{equation*}
 \end{enumerate}
\end{proof}


\begin{corol}\label{partiel:corol estim control partiel}
Let us suppose that for all $\varphi_0\in L^2(\Omega)^p$, the solution $\varphi$ 
to the adjoint System (\ref{partiel:obs syst adjoint}) satisfies 
the observability inequality (\ref{partiel:inegalite d observabilite lineraire2}). 
Then for all initial condition $y_0\in L^2(\Omega)^n$, there exists a control $u\in L^2(q_T)^m$ 
($q_T:=\omega\times(0,T)$) 
such that the solution $y$ to System (\ref{partiel:syst para lin equa contr intro}) satisfies 
$ \Pi_p y(T)\equiv0\mr{~in~}\Omega$ 
and 
\begin{equation}\label{partiel:estim control partiel}
 \|u\|_{L^2(q_T)^m}\leqslant \sqrt{C_{obs}}\|y_0\|_{L^2(\Omega)^n}.
\end{equation}
\end{corol}

The proof is classical and will be omitted (estimate (\ref{partiel:estim control partiel}) can be  obtained directly following the method developed 
in \cite{zuazua_fernandez_00}).

\section{Partial null controllability with constant coupling matrices}\label{partiel:section constant}

\hspace*{4mm}


Let us consider the system
\begin{equation}\label{partiel:syst contant case}
\begin{array}{l}
  \left\{\begin{array}{ll}
   \partial_t y=\Delta y+Ay+B\mathds{1}_{\omega}u&\mr{~in~}Q_T,\\
        y=0&\mr{~on}~\Sigma_T,\\
	y(0)=y_0 &~\mr{in}~\Omega,
        \end{array}
\right.
\end{array}
\end{equation}
where $y_0\in L^2(\Omega)^n$, $u\in L^2(Q_T)^m$, 
$ A\in\mc{L}(\mb{R}^{n})$ and $ B\in\mc{L}(\mb{R}^m;\mb{R}^n)$. 
Let the natural number  $s$ be defined by
 \begin{equation}\label{partiel:def s const one force}
  s:=\mr{rank} ~[A|B]
 \end{equation}
 and $X\subset \mb{R}^n$ be the linear space spanned by the columns of $[A|B]$. 

 In this section, we prove Theorem \ref{partiel:theorem contro coupl constant} in two steps. 
In subsection \ref{partiel:section one force}, we begin by studying 
the case where $B\in\mb{R}^n$ and the general case is considered in subsection 
\ref{partiel:section m forces}. 

All along this section, we will use the  lemma below which proof is straightforward.

 \begin{Lemme}\label{partiel:lemme prelim one force2}
 Let be $y_0\in L^2(\Omega)^n$, $u\in L^2(Q_T)^m$ and $P\in\mc{C}^1([0,T],\mc{L}(\mb{R}^n))$ 
 such that $P(t)$ is invertible for all $t\in[0,T]$. 
 Then  the change of variable  $w=P^{-1}(t)y$ transforms System (\ref{partiel:syst contant case})  
 into the equivalent system 
\begin{equation}\label{partiel:syst contant case w lemme one force}
\begin{array}{l}
  \left\{\begin{array}{ll}
   \partial_t w=\Delta w+C(t)w+D(t)\mathds{1}_{\omega}u&\mr{~in~}Q_T,\\
        w=0&\mr{~on}~\Sigma_T,\\
        w(0)=w_0 &~\mr{in}~\Omega,
        \end{array}
\right.
\end{array}
\end{equation}
with $w_0:=P^{-1}(0)y_0$, $C(t):=-P^{-1}(t)\partial_t P(t)+P^{-1}(t)AP(t)$ and $D(t):=P^{-1}(t)B$. 
Moreover
\begin{equation*}
\Pi_py(T)\equiv0\mr{~in~} \Omega~\Leftrightarrow~ \Pi_pP(T)w(T)\equiv0 \mr{~in~} \Omega. 
\end{equation*}
If $P$ is constant, we have
\begin{equation*}
[C|D]=P^{-1}[A|B]. 
\end{equation*}
\end{Lemme}


\subsection{One control force}\label{partiel:section one force}

\hspace*{4mm} In this subsection, we suppose that $A\in \mc{L}(\mb{R}^n)$,  
 $B\in \mb{R}^n$ and denote by $[A|B]=:(k_{ij})_{1\leqslant i,j\leqslant n}$ and $s:=\mr{rank}~[A|B]$. 
 We  begin with the following observation.

\begin{Lemme}\label{partiel:lemme one force sous matrice}
 $\{B,...,A^{s-1}B\}$ is a basis of $X$.
 \end{Lemme}

 \begin{proof}
 If $s=\mr{rank}~[A|B]=1$, then the conclusion of the lemma is clearly true, since $B\neq0$. Let $s\geqslant 2$. 
   Suppose to the contrary that $\{B,...,A^{s-1}B\}$ is not a basis of $X$, that is 
   for some $i\in \{0,...,s-2\}$ 
   the family $\{B,...,A^{i}B\}$ is linearly independent 
   and $A^{i+1}B\in\mr{span}(B,...,A^{i}B)$, that is 
   $A^{i+1}B=\sum_{k=0}^i\alpha_{k}A^{k}B$ 
   with $\alpha_0,...,\alpha_i\in\mb{R}$.  
   Multiplying by $A$ this expression, we deduce that  
   $A^{i+2}B\in\mr{span}(AB,...,A^{i+1}B)=\mr{span}(B,...,A^{i}B)$. 
 Thus, by induction, $A^lB\in\mr{span}(B,...,A^{i}B)$ 
   for all $l\in  \{ i+1,...,n-1\}$. 
 Then $\mr{rank} ~(B|AB|...|A^{n-1}B)=\mr{rank} ~(B|AB|...|A^{i}B)=i+1<s$, contradicting with (\ref{partiel:def s const one force}). 
 \end{proof}

\begin{proof}[Proof of Theorem \ref{partiel:theorem contro coupl constant}]
Let us  remark that 
\begin{equation}\label{partiel:p<s}
\mr{rank}~\Pi_p[A|B]=\mr{dim}~
 \Pi_p[A|B](\mb{R}^n)\leqslant   \mr{rank}~[A|B]=s. 
\end{equation}
Lemma \ref{partiel:lemme one force sous matrice} yields
\begin{equation}\label{partiel:rank B AB As-1B =s}
\mr{rank}~(B|AB|...|A^{s-1}B)= \mr{rank}~[A|B]=s.
\end{equation}
Thus, for all $l\in \{ s,s+1,..., n\} $ and  $i\in\{0,...,s-1\}$, 
there exist $\alpha_{l,i}$   such that 
\begin{equation}\label{partiel:comb lin A^s}
A^lB=\sum\limits_{i=0}^{s-1}\alpha_{l,i}A^iB. 
\end{equation}
Since,  for all  $l \in \{ s,..., n\} $, 
 $\Pi_pA^lB=\sum_{i=0}^{s-1}\alpha_{l,i}\Pi_pA^iB$, then
 \begin{equation}\label{partiel:rang PB PAB}
\mr{rank}~\Pi_p(B|AB|...|A^{s-1}B)=\mr{rank}~\Pi_p[A|B].
\end{equation}
We first prove in (a) that condition (\ref{partiel:cond cas constant une force}) is sufficient, and then in (b) that this condition is necessary. 

\hspace*{2mm}
(a) Sufficiency  part: 
Let us assume first that condition (\ref{partiel:cond cas constant une force}) holds.  
Then, using (\ref{partiel:rang PB PAB}), we have
 \begin{equation}\label{partiel:rang PB PAB p}
\mr{rank}~\Pi_p(B|AB|...|A^{s-1}B)=p.
\end{equation}
Let be $y_0\in L^2(\Omega)^n$. 
We will study the $\Pi_p$-null controllability of System (\ref{partiel:syst contant case}) 
according to the values of $p$ and $s$. 

\begin{enumerate}[\textit{Case} 1]
 \item : $p=s$. 
 The idea is to find an appropriate change of variable  $P$ to the solution $y$ to System (\ref{partiel:syst contant case}).
  More precisely, we would like   the new variable $w:=P^{-1}y$ to be the solution to a cascade system  
  and then, apply  Theorem \ref{partiel:theo cascade}. So 
  let us define, for all $t\in [0,T]$,
\begin{equation}\label{partiel:def M p=s}
P(t):=(B|AB|...|A^{s-1}B|P_{s+1}(t)|...|P_n(t)), 
\end{equation}
 where, for all $l \in \{ s+1,..., n\} $,  $P_l(t)$ 
is the solution in $\mc{C}^{1}([0,T])^n$ to the system of ordinary differential equations
\begin{equation}\label{partiel:system K_2 p=s}
 \left\{\begin{array}{l}
       \partial_tP_l(t)=AP_l(t)\mr{~in~}[0,T],\\
       P_l(T)=e_{l}.
        \end{array}
\right.
\end{equation}
Using (\ref{partiel:def M p=s}) and (\ref{partiel:system K_2 p=s}), we can write 
\begin{equation}\label{partiel:definition K(T) p=s}
 P(T)=\left(\begin{array}{cc}
        P_{11}&0\\
        P_{21}&I_{n-s}
            \end{array}
\right),
\end{equation}
where $P_{11}:=\Pi_p(B|AB|...|A^{s-1}B)\in \mc{L}(\mb{R}^s)$,  
$P_{21}\in\mc{L}(\mb{R}^s,\mb{R}^{n-s})$ and $I_{n-s}$ is the identity matrix of size $n-s$. 
Using (\ref{partiel:rang PB PAB p}), $P_{11}$ is invertible and thus $P(T)$ also. 
Furthermore, since $P(t)$ is an element of $\mc{C}^1([0,T],\mc{L}(\mb{R}^n))$ 
continuous in time on the time interval $[0,T]$, 
there exists $T^*\in [0,T)$  such that 
$P(t)$ is invertible for all $t\in[T^*,T]$. \\
\hspace*{4mm} Let us suppose first that $T^*=0$. Since $P(t)$ is 
an element of $\mc{C}^1([0,T],\mc{L}(\mb{R}^n))$ and invertible, 
in view of Lemma \ref{partiel:lemme prelim one force2}:  
for a fixed 
control $u\in L^2(Q_T)$, 
$y$ is the solution to System (\ref{partiel:syst contant case}) 
if and only if $w:=P(t)^{-1}y$ is the solution to System 
(\ref{partiel:syst contant case w lemme one force}) 
where  $C$, $D$ are given by 
\begin{equation*}
 C(t):=-P^{-1}(t)\partial_t P(t)+P^{-1}(t)AP(t)\mr{~~~and~~~} D(t):=P^{-1}(t)B,
\end{equation*}
for all $t\in [0,T]$. Using (\ref{partiel:comb lin A^s}) and (\ref{partiel:system K_2 p=s}), we obtain
\begin{equation}\label{partiel:equation K p=s}
 \left\{\begin{array}{ll}
      -\partial_tP(t)+AP(t)=(AB|...|A^sB|0|...|0)=P(t)\left(
  \begin{array}{cc}
   C_{11}&0\\
   0&0
  \end{array}
  \right)&\mr{~in~}[0,T],\\
      P(t)e_1=B&\mr{~in~}[0,T],
        \end{array}
\right.
\end{equation} 
 where
 \begin{equation}\label{partiel:def C11 reciproque}
 C_{11}:=\left(
  \begin{array}{ccccc}
  0&0&0&\ldots&\alpha_{s,0}\\
  1&0&0&\ldots&\alpha_{s,1}\\
  0&1&0&\ldots&\alpha_{s,2}\\
  \vdots&\vdots&\ddots&\ddots&\vdots\\
  0&0&\ldots&1&\alpha_{s,s-1}
  \end{array}
  \right)\in \mc{L}(\mb{R}^s).
\end{equation}
Then 
\begin{equation}\label{partiel:def C p=s}
 C(t)=\left(
  \begin{array}{cc}
   C_{11}&0\\
   0&0
  \end{array}
  \right)\mr{~and~}D(t)=e_1.
\end{equation}
Using Theorem \ref{partiel:theo cascade}, there exists 
$u\in L^2(Q_T)$ such that the solution to System (\ref{partiel:syst contant case w lemme one force}) 
satisfies  $w_1(T)\equiv...\equiv w_s(T)\equiv0$ in $\Omega$. Moreover, using (\ref{partiel:definition K(T) p=s}), we have 
 \begin{equation*}
\Pi_sy(T)=(y_{1}(T),...,y_{s}(T))=P_{11}(w_1(T),...,w_s(T))\equiv0\mr{~in~}\Omega. 
\end{equation*}

 \hspace*{4mm} If now $T^*\neq 0$, let $\overline{y}$ be the solution in $W(0,T^*)^n$ 
 to System (\ref{partiel:syst contant case}) with the initial condition $\overline{y}(0)=y_0$ 
 in $\Omega$ and the control $u\equiv0$ in $\Omega\times(0,T^*)$. 
 We use the same argument as above to prove that 
  System (\ref{partiel:syst contant case}) is $\Pi_s$-null controllable on the time interval $[T^*,T]$.  
 Let $v$ be a control in $L^2(\Omega\times(T^*,T))$ such that the solution $z$ in $W(T^*,T)^n$ to System (\ref{partiel:syst contant case}) 
 with the initial condition $z(T^*)=\overline{y}(T^*)$ in $\Omega$ and the control $v$ satisfies $\Pi_sz(T)\equiv0$ in $\Omega$. 
 Thus if we define $y$ and $u$ as follows 
 \begin{equation*}
 (y,u):=\left\{\begin{array}{l}
       (\overline{y},0)\mr{~if~}t\in [0,T^*],\\
       (z,v)\mr{~if~}t\in [T^*,T],
          \end{array}
\right.
\end{equation*}
then, for this control $u$, $y$ is the solution in $W(0,T)^n$ to System (\ref{partiel:syst contant case}). 
Moreover $y$ 
 satisfies 
 \begin{equation*}
  \Pi_sy(T)\equiv0 \mr{~in~} \Omega.
 \end{equation*}

 \item : $p<s$. 
 In order to use Case 1, 
we would like  to apply an appropriate change of variable $Q$ to the solution $y$ to System (\ref{partiel:syst contant case}).   
If we denote by $[A|B]=:(k_{ij})_{ij}$, equalities (\ref{partiel:rank B AB As-1B =s}) and (\ref{partiel:rang PB PAB p}) can be rewritten
\begin{equation*}
             \mr{rank}~\left(\begin{array}{ccc}k_{11}&\cdots&k_{1s}\\
                \vdots&&\vdots\\
                k_{n1}&\cdots&k_{ns}
               \end{array}\right)=s
                \mr{~~~and~~~}      
               \mr{rank}~\left(\begin{array}{ccc}k_{11}&\cdots&k_{1s}\\
                \vdots&&\vdots\\
                k_{p1}&\cdots&k_{ps}
               \end{array}\right)=p.
\end{equation*}
Then there exist distinct natural numbers $\lambda_{p+1},...,\lambda_{s}$ 
such that  $\{\lambda_{p+1},...,\lambda_{s}\}\subset \{ p+1,...,n\}$ 
and 
\begin{equation}\label{partiel:cas const sous matrice}
\mr{rank}\left(\begin{array}{ccc}k_{11}&\cdots&k_{1s}\\
                \vdots&&\vdots\\
                k_{p1}&\cdots&k_{ps}\\
                k_{\lambda_{p+1}1}&\cdots&k_{\lambda_{p+1}s}\\
                \vdots&&\vdots\\
                k_{\lambda_{s}1}&\cdots&k_{\lambda_{s}s}\\
               \end{array}\right)=s.
\end{equation}
Let  $Q$ be the  matrix defined  by
\begin{equation*}
Q :=(e_{1}|...|e_p|e_{\lambda_{p+1}}|...|e_{\lambda_n})^{\mr{t}},
\end{equation*}
where 
$\{\lambda_{s+1},...,\lambda_{n}\}:
=\{ p+1,...,n\}\backslash\{\lambda_{p+1},...,\lambda_s\}$. 
 $Q$ is invertible, 
so taking   $w:=P^{-1}y$ with $P:=Q^{-1}$, 
for a fixed control $u$ in $L^2(Q_T)$, 
$y$ is solution to System (\ref{partiel:syst contant case}) 
if and only if $w$ is solution to System (\ref{partiel:syst contant case w lemme one force}) 
where $w_0:=Qy_0$, $ C:=QAQ^{-1}\in\mc{L}(\mb{R}^{n})$ and 
$D:=QB\in\mc{L}(\mb{R};\mb{R}^n)$. 
Moreover there holds
\begin{equation*}
 [C|D]=Q[A|B].
\end{equation*}
Thus, 
equation (\ref{partiel:cas const sous matrice}) yields
\begin{equation*}
 \mr{rank}~\Pi_s[C|D]=\mr{rank}~\Pi_sQ[A|B]=
 \mr{rank}~\left(\begin{array}{ccc}k_{11}&\cdots&k_{1n}\\
                \vdots&&\vdots\\
                k_{p1}&\cdots&k_{pn}\\
                k_{\lambda_{p+1}1}&\cdots&k_{\lambda_{p+1}n}\\
                \vdots&&\vdots\\
                k_{\lambda_{s}1}&\cdots&k_{\lambda_{s}n}\\
               \end{array}\right)=s.
\end{equation*}
Since $\mr{rank}~[C|D]=\mr{rank}~[A|B]=s$, we proceed as in Case 1 forward deduce that 
System (\ref{partiel:syst contant case w lemme one force})  is $\Pi_s$-null controllable, 
that is there exists a control $u\in L^2(Q_T)$ such that the solution $w$ to 
System (\ref{partiel:syst contant case w lemme one force})  satisfies 
\begin{equation*}
\Pi_sw(T)\equiv 0\mr{~ in~} \Omega.
\end{equation*}
Moreover the matrix $Q$ can be rewritten
\begin{equation*}
Q=\left(\begin{array}{cc}I_p&0\\
                0&Q_{22}\\
               \end{array}\right),
\end{equation*}
where $Q_{22}\in\mc{L}(\mb{R}^{n-p})$.
Thus 
\begin{equation*}
 \Pi_py(T)=\Pi_pQy(T)=\Pi_pw(T)\equiv 0\mr{~ in~} \Omega.
\end{equation*}

 \end{enumerate}

\hspace*{2mm}(b) Necessary part: Let us denote by $[A|B]=:(k_{ij})_{ij}$. 
 We suppose now that (\ref{partiel:cond cas constant une force}) is not satisfied: 
 there exist $\overline{p}\in\{ 1,...,p\}$ and $\beta_i$ 
 for all $i\in\{1,...,p\}\backslash\{\overline{p}\}$  
  such that $k_{\overline{p}j}=\sum\limits_{i=1,i\neq \overline{p}}^p\beta_ik_{ij}$ 
  for all $j\in\{ 1,...,s\}$. 
The idea is to find a change of variable $w:=Qy$ 
that allows to handle more easily our system. 
We will achieve this in three steps starting from the simplest situation.
 \begin{enumerate}[Step 1.]
  \item Let us suppose first that 
  \begin{equation}\label{partiel:one force hyp 1}
  k_{11}=...=k_{1s}=0 \mr{~~~and~~~} 
  \mr{rank}~\left(\begin{array}{ccc}k_{21}&\cdots&k_{2s}\\
                \vdots&&\vdots\\
                k_{s+1,1}&\cdots&k_{s+1,s}
               \end{array}\right)=s.
  \end{equation}
We want to prove that, for some initial condition $y_0\in L^2(\Omega)^n$, a control  $u\in L^2(Q_T)$ 
cannot be found such that the solution to System (\ref{partiel:syst contant case}) satisfies 
$y_1(T)\equiv 0$ in $\Omega$. Let us consider the matrix $P\in\mc{L}(\mb{R}^n)$ defined by
  \begin{equation}\label{partiel:def P one force}
   P:=(B|...|A^{s-1}B|e_1|e_{s+2}|...|e_n).
  \end{equation}
Using the assumption (\ref{partiel:one force hyp 1}), $P$ is invertible. 
Thus, in view of Lemma \ref{partiel:lemme prelim one force2}, 
for a fixed 
control $u\in L^2(Q_T)$, 
$y$ is a solution to System (\ref{partiel:syst contant case}) 
if and only if $w:=P^{-1}y$ is a solution to System 
(\ref{partiel:syst contant case w lemme one force}) 
where  $C$, $D$ are given by 
 $C:=P^{-1}AP$ and $D:=P^{-1}B$.  
Using (\ref{partiel:comb lin A^s}) we remark that
 \begin{equation*}
  A(B|AB|...|BA^{s-1})=(B|AB|...|BA^{s-1})\left(\begin{array}{c}C_{11}\\0\end{array}\right),
 \end{equation*}
 with $C_{11}$ defined in (\ref{partiel:def C11 reciproque}).
 Then $C$ can be rewritten as
\begin{equation}\label{partiel:def C reciproque}
C =\left(\begin{array}{cc}
        C_{11}&C_{12}\\
        0&C_{22}
       \end{array}\right),
\end{equation}
where $ C_{12}\in \mc{L}(\mb{R}^{n-s},\mb{R}^s)$ and $ C_{22}\in \mc{L}(\mb{R}^{n-s})$. 
Furthermore
\begin{equation*}
D=P^{-1}B=P^{-1}Pe_1=e_1.
\end{equation*}
and with the Definition (\ref{partiel:def P one force}) of $P$ we get
 \begin{equation*}
y_1(T)=w_{s+1}(T)\mr{~in~}\Omega.
\end{equation*}
Thus we need only to prove that there exists $w_0\in L^2(\Omega)^n$ 
such that we cannot find  a control $u\in L^2(Q_T)$ 
with the corresponding solution $w$ to System  (\ref{partiel:syst contant case w lemme one force}) 
satisfying $w_{s+1}(T)\equiv0$  in $\Omega$.  
Therefore we apply Proposition \ref{partiel:property equivalence contr obs} and 
prove that the observability inequality \eqref{partiel:inegalite d observabilite lineraire2} 
can not be satisfied. 
More precisely, 
for all $w_0\in L^2(\Omega)^n$, there exists a control $u\in L^2(Q_T)$ 
such that the solution to System  (\ref{partiel:syst contant case w lemme one force}) satisfies $w_{s+1}(T)\equiv0$  in $\Omega$, 
if and only if there exists $C_{obs}>0$ such that 
for all $\varphi_{s+1}^0\in L^2(\Omega)$ the solution to the adjoint system
  \begin{equation}\label{partiel:syst constant dual rg=0}
\begin{array}{l}
  \left\{\begin{array}{rcll}
-\partial_t\varphi&=&\Delta \varphi
+\left(\begin{array}{cc}
        C_{11}^*&0\\
        C_{12}^*&C_{22}^*
       \end{array}\right)
\varphi&\mr{~in~}Q_T,\\
        \varphi&=&0&\mr{~on}~\Sigma_T,\\
	\varphi(T)&=&(0,...,0,\varphi^0_{s+1},0,...,0)^{\mr{t}}=e_{s+1}\varphi^0_{s+1} &~\mr{in}~\Omega
        \end{array}
\right.
\end{array}
\end{equation}
 satisfies the observability inequality 
 \begin{equation}\label{partiel:ine obs rg=0}
  \displaystyle\int_{\Omega}\varphi(0)^2\,dx
  \leqslant C_{obs}\displaystyle\int_{\omega\times(0,T)}\varphi_1^2\,dx\,dt.
 \end{equation}
But for all  $\varphi_{s+1}^0\not\equiv0$ in $\Omega$, 
the inequality (\ref{partiel:ine obs rg=0}) is not satisfied. 
Indeed,  we remark first that, since $\varphi_1(T)=...=\varphi_s(T)=0$ in $\Omega$, 
 we have $\varphi_1=...=\varphi_s=0$ in $Q_T$, so that 
 $\int_{\omega\times(0,T)}\varphi_1^2\,dx=0,$  
 while, if we choose $\varphi_{s+1}^0\not\equiv 0$ in $\Omega$, 
  using the results on backward uniqueness
for this type of parabolic system  (see \cite{Ghidaglia86}), 
we have clearly $(\varphi_{s+1}(0),...,\varphi_n(0))\not\equiv0$ in $\Omega$.

  \item Let us suppose only that $k_{11}=...=k_{1s}=0$. 
 Since $\mr{rank}~(B|...|A^{s-1}B)=s$, 
 there exists distinct $\lambda_1,...,\lambda_s\in\{2,...,n\}$ such that  
  \begin{equation*}
  \mr{rank}~\left(\begin{array}{ccc}k_{\lambda_{1},1}&\cdots&k_{\lambda_{1},s}\\
                \vdots&&\vdots\\
                k_{\lambda_{s},1}&\cdots&k_{\lambda_{s},s}
               \end{array}\right)=s.
  \end{equation*}
Let us consider the following matrix 
    \begin{equation*}
   Q:=(e_1|e_{\lambda_1}|...|e_{\lambda_{n-1}})^{\mr{t}},
  \end{equation*}
  where $\{\lambda_{s+1},...,\lambda_{n-1}\}=\{2,...,n\}\backslash\{\lambda_{1},...,\lambda_{s}\}$. 
Thus, for $P:=Q^{-1}$, again, 
for a fixed 
control $u\in L^2(Q_T)$, $y$ is a solution to System (\ref{partiel:syst contant case}) 
if and only if $w:=P^{-1}y$ is a solution to System 
(\ref{partiel:syst contant case w lemme one force}) 
where  $C$, $D$ are given by 
$C:=QAQ^{-1}$ and $D:=QB$. 
 Moreover, we have
 \begin{equation*}
  [C|D]=Q[A|B].
 \end{equation*}
If we note $(\tilde{k}_{ij})_{ij}:=[C|D]$, 
this implies $\tilde{k}_{11}=...=\tilde{k}_{1s}=0$ and
  \begin{equation*}
    \mr{rank}~\left(\begin{array}{ccc}\tilde{k}_{21}&\cdots&\tilde{k}_{2s}\\
                \vdots&&\vdots\\
                \tilde{k}_{s+1,1}&\cdots&\tilde{k}_{s+1,s}
               \end{array}\right)=
  \mr{rank}~\left(\begin{array}{ccc}k_{\lambda_{1}1}&\cdots&k_{\lambda_{1}s}\\
                \vdots&&\vdots\\
                k_{\lambda_{s},1}&\cdots&k_{\lambda_{s},s}
               \end{array}\right)=s.
  \end{equation*}
Proceeding as in Step 1 for $w$, there exists an initial condition $w_0$ such that for all control $u$ in $L^2(Q_T)$ the solution 
$w$ to System (\ref{partiel:syst contant case w lemme one force}) satisfies  $w_1(T) \not\equiv 0$ in $\Omega$. 
Thus, for the initial condition $y_0:=Q^{-1}w_0$, 
for all control $u$ in $L^2(Q_T)$, the solution $y$ to System (\ref{partiel:syst contant case}) satisfies
\begin{equation*}
y_1(T)=w_1(T) \not\equiv 0 \mr{~in~}\Omega.
\end{equation*}

  \item Without loss of generality, 
  we can  suppose that there exists $\beta_i$ for all $i\in\{ 2,...,p\}$ such that 
  $k_{1j}=\sum\limits_{i=2}^p\beta_i k_{ij}$ for all $j\in\{ 1,...,s\}$ 
  (otherwise a permutation of lines leads to this case). 
  Let us define  the following matrix 
    \begin{equation*}
   Q:=\left((e_1-\sum\limits_{i=2}^p\beta_ie_i)|e_{2}|...|e_{n}\right)^{\mr{t}}.
  \end{equation*}
  Thus, for $P:=Q^{-1}$, again, 
for a fixed initial condition $y_0\in L^2(\Omega)^n$ and a 
control $u\in L^2(Q_T)$, consider  System  (\ref{partiel:syst contant case w lemme one force})
with  $w:=P^{-1}y$, $y$ being a solution to System (\ref{partiel:syst contant case}).
%
  We remark that if we denote by $(\tilde{k}_{ij}):=[C|D]$, we have $\tilde{k}_{11}=...=\tilde{k}_{1s}=0$. 
  Applying step 2 to $w$, there exists an initial condition $w_0$ such that for all control $u$ in $L^2(Q_T)$ the solution 
$w$ to System (\ref{partiel:syst contant case w lemme one force}) satisfies 
\begin{equation}\label{partiel:wi(T)=0}
w_1(T) \not\equiv 0 \mr{~in~}\Omega.  
\end{equation}
Thus, with the definition of $Q$, for all control $u$ in $L^2(Q_T)$ the solution $y$ to System (\ref{partiel:syst contant case}) satisfies
\begin{equation*}
w_1(T)=y_1(T)-\sum\limits_{i=2}^p\beta_iy_i(T) \mr{~in~}\Omega.
\end{equation*}
Suppose $\Pi_py(T) \equiv 0$ in $\Omega$, then $w_1(T) \equiv 0$ in $\Omega$ and 
this contradicts (\ref{partiel:wi(T)=0}).  


 \end{enumerate}

  As a consequence of Proposition \ref{partiel:property equivalence contr obs}, 
 the $\Pi_p$-null controllability implies the $\Pi_p$-approximate controllability 
of  System  (\ref{partiel:syst contant case w lemme one force}). 
If now Condition (\ref{partiel:cond cas constant une force}) is not satisfied, 
as for the $\Pi_p$-null controllability, 
we can find a solution to System (\ref{partiel:syst constant dual rg=0}) 
such that $\phi_1\equiv 0$ in $\omega\times(0,T)$ and $\phi\not\equiv0$ in $Q_T$ and 
we conclude again with Proposition \ref{partiel:property equivalence contr obs}.

\end{proof}

\subsection{$m$-control forces}\label{partiel:section m forces}

\hspace*{4mm} In this subsection, we will suppose that $A\in \mc{L}(\mb{R}^n)$ 
and $B\in \mc{L}(\mb{R}^m,\mb{R}^n)$. We denote by 
$B=:(b^1|...|b^m)$. 
To prove Theorem \ref{partiel:theorem contro coupl constant}, we will use the following 
lemma which can be found in \cite{ammar2009generalization}. 

\begin{Lemme}\label{partiel:lemme constant base}
There exist $r\in \{ 1,...,s\}$ 
and sequences $\{l_j\}_{1\leqslant j\leqslant r}\subset\{1,...,m\}$ 
and  $\{s_j\}_{1\leqslant j\leqslant r}\subset\{1,...,n\}$ 
with $\sum_{j=1}^rs_j=s$, such that 
\begin{equation*}
 \mc{B}:=\bigcup\limits_{j=1}^r\{b^{l_j},Ab^{l_j},...,A^{s_j-1}b^{l_j}\}
\end{equation*}
is a basis of X. 
Moreover, for every $1\leqslant j\leqslant r$, 
there exist $\alpha_{k,s_j}^i\in \mb{R}$ 
for $1\leqslant i\leqslant j$ and $1\leqslant k\leqslant s_j$ such that
\begin{equation}\label{partiel:lemme constant base expression As_j}
 A^{s_j}b^{l_j}=
 \sum\limits_{i=1}^j\left( \alpha_{1,s_j}^ib^{l_i}+\alpha_{2,s_j}^iAb^{l_i}+
 ...+\alpha_{s_i,s_j}^iA^{s_i-1}b^{l_i}\right).
\end{equation}
\end{Lemme}

\begin{proof}[Proof of Theorem \ref{partiel:theorem contro coupl constant}]
Consider the basis $\mc{B}$ of $X$ given by Lemma \ref{partiel:lemme constant base}. 
Note that 
\begin{equation*}
\mr{rank}~\Pi_p[A|B]=\mr{dim}~
 \Pi_p[A|B](\mb{R}^n)\leqslant   \mr{rank}~[A|B]=s. 
\end{equation*}
If  $M$ is the matrix whose columns are the elements of $\mc{B}$, i.e.
 \begin{equation*}
  M=(m_{ij})_{ij}:=\left(b^{l_1}|Ab^{l_1}|...|A^{s_1-1}b^{l_1}|...
  |b^{l_r}|Ab^{l_r}|...|A^{s_r-1}b^{l_r}\right),
 \end{equation*}
we can remark that 
\begin{equation}\label{partiel:rank pi_pM m forces}
 \mr{rank}~\Pi_pM=\mr{rank}~\Pi_p[A|B].
\end{equation}
Indeed, relationship (\ref{partiel:lemme constant base expression As_j}) yields
\begin{equation*}
 \Pi_pA^{s_j}b^{l_j}=
 \sum\limits_{i=1}^j\left( \alpha_{1,s_j}^i\Pi_pb^{l_i}+\alpha_{2,s_j}^i\Pi_pAb^{l_i}+
 ...+\alpha_{s_i,s_j}^i\Pi_pA^{s_i-1}b^{l_i}\right).
\end{equation*}
We first prove in (a) that condition (\ref{partiel:cond cas constant une force}) is sufficient, and then in (b) that this condition is necessary. 

\hspace*{2mm}
(a) Sufficiency part: Let us suppose first that (\ref{partiel:cond cas constant une force}) is satisfied. Let be $y_0\in L^2(\Omega)^n$. 
We will prove that we need only $r$ forces to control System 
(\ref{partiel:syst contant case}). More precisely, we will study the 
$\Pi_p$-null controllability of the system 
\begin{equation}\label{partiel:syst contant case r force}
\begin{array}{l}
  \left\{\begin{array}{ll}
   \partial_t y=\Delta y+Ay+\tilde{B}
   1_{\omega}v&\mr{~in~}Q_T,\\
        y=0&\mr{~on}~\Sigma_T,\\
	y(0)=y_0 &~\mr{in}~\Omega,
        \end{array}
\right.
\end{array}
\end{equation}
where 
$\tilde{B}=(b^{l_1}|b^{l_2}|\cdots|b^{l_r})
\in \mc{L}(\mb{R}^r,\mb{R}^n)$.
Using (\ref{partiel:cond cas constant une force}) and (\ref{partiel:rank pi_pM m forces}), we have
\begin{equation}\label{partiel:rank P11=p p=s m force}
 \mr{rank}~\Pi_p(b^{l_1}|Ab^{l_1}|...|A^{s_1-1}b^{l_1}|...
  |b^{l_r}|Ab^{l_r}|...|A^{s_r-1}b^{l_r})=p. 
\end{equation}

\begin{enumerate}[\textit{Case} 1]
 \item : $p=s$. 
 As in the case of one control force, we want to apply a change of variable 
 $P$ to the solution $y$ to System (\ref{partiel:syst contant case r force}). 
 Let us define for all $t\in [0,T]$ the following matrix
  \begin{equation}\label{partiel:def P m force p=s}
  P(t):=(b^{l_1}|Ab^{l_1}|...|A^{s_1-1}b^{l_1}|...
  |b^{l_r}|Ab^{l_r}|...|A^{s_r-1}b^{l_r}|P_{s+1}(t)|...|P_n(t))\in \mc{L}(\mb{R}^n),
 \end{equation}
 where for all $l\in \{ s+1,...,n\}$, $P_l$ is solution in $\mc{C}^1([0,T])^n$ to the 
 system of ordinary differential equations
 \begin{equation}\label{partiel:system K_2 p=s m foces}
  \left\{\begin{array}{l}
          \partial_tP_l(t)=AP_l(t)\mr{~in~}[0,T],\\
          P_l(T)=e_l.
         \end{array}
\right.
 \end{equation}
Using (\ref{partiel:def P m force p=s}) and (\ref{partiel:system K_2 p=s m foces}) we have 
\begin{equation}\label{partiel:definition K(T) p=s m foces}
 P(T)=\left(\begin{array}{cc}
        P_{11}&0\\
        P_{21}&I_{n-s}
            \end{array}
\right),
\end{equation}
where $P_{11}:=\Pi_s(b^{l_1}|Ab^{l_1}|...|A^{s_1-1}b^{l_1}|...
  |b^{l_r}|Ab^{l_r}|...|A^{s_r-1}b^{l_r})\in \mc{L}(\mb{R}^s)$ and $P_{21}\in\mc{L}(\mb{R}^{n-s},\mb{R}^s)$. 
From (\ref{partiel:rank P11=p p=s m force}), $P_{11}$ and thus $P(T)$ are invertible. 
Furthermore, since $P$ is continuous on $[0,T]$, 
there exists a $T^*\in [0,T)$  such that 
$P(t)$ is invertible for all $t\in[T^*,T]$. 

\hspace*{4mm} We suppose first that $T^*=0$. Since $P$ is invertible and continuous on  $[0,T]$, 
for a fixed 
control $v\in L^2(Q_T)^r$, 
$y$ is the solution to System (\ref{partiel:syst contant case r force}) 
if and only if $w:=P(t)^{-1}y$ is the solution to System 
(\ref{partiel:syst contant case w lemme one force}) 
where  $C$, $D$ are given by 
\begin{equation*}
 C(t):=-P^{-1}(t)\partial_t P(t)+P^{-1}(t)AP(t)\mr{~~~and~~~} D(t):=P^{-1}(t)\tilde{B},
\end{equation*}
for all $t\in[0,T]$. 
Using (\ref{partiel:lemme constant base expression As_j}) 
and (\ref{partiel:system K_2 p=s m foces}), we obtain
\begin{equation}\label{partiel:equation K p<s}
 \left\{\begin{array}{rcll}
      -\partial_tP(t)+AP(t)&=&(Ab^{l_1}|A^2b^{l_1}|...|A^{s_1}b^{l_1}|...
  |Ab^{l_r}|A^2b^{l_r}|...|A^{s_r}b^{l_r}|0|...|0),\\&=&P(t)\left(
  \begin{array}{cc}
   \tilde{C}_{11}&0\\
   0&0
  \end{array}
  \right)&\mr{~in~}[0,T],\\
      P(t)e_{S_i}&=&b^{l_i}&\mr{~in~}[0,T],
        \end{array}
\right.
\end{equation} 
where $S_i=1+\sum_{j=1}^{i-1}s_j$ for $i\in\{1,...,r\}$, 
  \begin{equation}\label{partiel:def C11 reciproque m force}
  \tilde{C}_{11}:=\left(
  \begin{array}{cccc}
   C_{11}&C_{12}&\cdots&C_{1r}\\
   0&C_{22}&\cdots&C_{2r}\\
    \vdots&\vdots&\ddots&\vdots\\
   0&0&\cdots&C_{rr}\\
  \end{array}
  \right)\in\mc{L}(\mb{R}^s)
  \end{equation}
and for $1\leqslant i\leqslant j\leqslant r$
the matrices 
 $C_{ij}\in\mc{L}(\mb{R}^{s_j},\mb{R}^{s_i})$ 
 are given by
 \begin{equation}\label{partiel:definition Cij}  
 C_{ii}:=\left(
  \begin{array}{ccccc}
  0&0&0&\ldots&\alpha_{1,s_i}^i\\
  1&0&0&\ldots&\alpha_{2,s_i}^i\\
  0&1&0&\ldots&\alpha_{3,s_i}^i\\
  \vdots&\vdots&\ddots&\ddots&\vdots\\
  0&0&\ldots&1&\alpha_{s_i,s_i}^i
  \end{array}
  \right) 
   ~~\mr{~~and~}
C_{ij}:=\left(
  \begin{array}{ccccc}
  0&0&0&\ldots&\alpha_{1,s_j}^i\\
  0&0&0&\ldots&\alpha_{2,s_j}^i\\
  0&0&0&\ldots&\alpha_{3,s_j}^i\\
  \vdots&\vdots&\ddots&\ddots&\vdots\\
  0&0&\ldots&0&\alpha_{s_i,s_j}^i
  \end{array}
  \right)~\mr{for}~j>i.
\end{equation}
Then 
\begin{equation}\label{partiel:def C p=s mforces}
 C(t)=\left(
  \begin{array}{cc}
   \tilde{C}_{11}&0\\
   0&0
  \end{array}
  \right)\mr{~and~}D(t)=(e_{S_1}|...|e_{S_r}).
\end{equation}
Using Theorem \ref{partiel:theo cascade}, there exists 
$v\in L^2(Q_T)^r$ such that the solution to System (\ref{partiel:syst contant case w lemme one force}) 
satisfies  $w_1(T)=...=w_s(T)\equiv0$ in $\Omega$. Moreover, using (\ref{partiel:definition K(T) p=s m foces}), 
we have 
 \begin{equation*}
\Pi_sy(T)=(y_{1}(T),...,y_{s}(T))
=P_{11}(w_1(T),...,w_s(T))\equiv0\mr{~in~}\Omega. 
\end{equation*}
 \hspace*{4mm} If now $T^*\neq 0$, we conclude as in the proof of 
 Theorem \ref{partiel:theorem contro coupl constant} with one force (see $\S$ \ref{partiel:section one force}).
 
 \item : $p<s$.
The proof is a direct adaptation of 
  the proof of Theorem \ref{partiel:theorem contro coupl constant} with one force, it is possible 
to find a change of variable in order to get back to the situation of Case 1 
(see $\S$ \ref{partiel:section one force}).
\end{enumerate}

\hspace*{2mm}
(b) Necessary part: If  (\ref{partiel:cond cas constant une force})  is not satisfied, 
 there exist $\overline{p}\in\{ 1,...,p\}$ and, for all $i\in\{1,...,p\}\backslash\{\overline{p}\}$, scalars $\beta_i$   
  such that $m_{\overline{p}j}=\sum\limits_{i=1,i\neq \overline{p}}^p\beta_im_{ij}$ 
  for all $j\in\{ 1,...,s\}$. 
 As previously, without loss of generality, we can suppose that 
  \begin{equation}\label{partiel:hyp r forces}
  m_{11}=...=m_{1s}=0 \mr{~~~and~~~} 
  \mr{rank}~\left(\begin{array}{ccc}m_{21}&\cdots&m_{2s}\\
                \vdots&&\vdots\\
                m_{s+1,1}&\cdots&m_{s+1,s}
               \end{array}\right)=s
  \end{equation}
(otherwise a permutation of lines leads to this case). Let us consider the matrix $P$ defined by
  \begin{equation}\label{partiel:def P one force rec}
   P:=(b^{l_1}|Ab^{l_1}|...|A^{s_1-1}b^{l_1}|...
  |b^{l_r}|Ab^{l_r}|...|A^{s_r-1}b^{l_r}|e_1|e_{s+2}|...|e_n).
  \end{equation}
Relationship ensures (\ref{partiel:hyp r forces}) that $P$ is invertible. 
Thus, again, 
for a fixed 
control $u\in L^2(Q_T)^m$, 
$y$ is the solution to System (\ref{partiel:syst contant case}) 
if and only if $w:=P^{-1}y$ is the solution to System 
(\ref{partiel:syst contant case w lemme one force}) 
where  $C$, $D$ are given by 
 $C:=P^{-1}AP$ and $D:=P^{-1}B$. 
Using (\ref{partiel:lemme constant base expression As_j}), we remark that
 \begin{equation*}\begin{array}{c}
 A(b^{l_1}|Ab^{l_1}|...|A^{s_1-1}b^{l_1}|...
  |b^{l_r}|Ab^{l_r}|...|A^{s_r-1}b^{l_r})\\
  =(Ab^{l_1}|A^2b^{l_1}|...|A^{s_1}b^{l_1}|...
  |Ab^{l_r}|A^2b^{l_r}|...|A^{s_r}b^{l_r})
 =P \left(\begin{array}{c}\tilde{C}_{11}\\0\end{array}\right),
 \end{array}\end{equation*}
where $\tilde{C}_{11}$ is defined in (\ref{partiel:def C11 reciproque m force}).
 Then $C$ can be written as
\begin{equation}\label{partiel:def C reciproque m forces}
C =\left(\begin{array}{cc}
        \tilde{C}_{11}&\tilde{C}_{12}\\
        0&\tilde{C}_{22}
       \end{array}\right),
\end{equation}
where $ \tilde{C}_{12}\in \mc{L}(\mb{R}^s,\mb{R}^{n-s})$ 
and $ \tilde{C}_{22}\in \mc{L}(\mb{R}^{n-s})$. 
Furthermore, the matrix $D$ can be written
\begin{equation*}
D=\left(\begin{array}{c}
         D_1\\
         0
        \end{array}
\right),
\end{equation*}
where $D_1\in\mc{L}(\mb{R}^m,\mb{R}^s)$. Using  (\ref{partiel:def P one force rec}), we get
 \begin{equation*}
y_1(T)=w_{s+1}(T)\mr{~in~}\Omega.
\end{equation*}
Thus, we need only to prove that there exists $w_0\in L^2(\Omega)^n$ 
such that we cannot find  a control $u\in L^2(Q_T)^m$ 
with the corresponding solution $w$ to System  (\ref{partiel:syst contant case w lemme one force}) 
satisfying $w_{s+1}(T)\equiv0$  in $\Omega$.  
Therefore we apply Proposition \ref{partiel:property equivalence contr obs} and 
prove that the observability inequality \eqref{partiel:inegalite d observabilite lineraire2} 
can not be satisfied. 
More precisely, 
for all $w_0\in L^2(\Omega)^n$, there exists a control $u\in L^2(Q_T)^m$ 
such that the solution $w$ to System  (\ref{partiel:syst contant case w lemme one force}) 
satisfies $w_{s+1}(T)\equiv0$  in $\Omega$, 
if and only if there exists $C_{obs}>0$ such that 
for all $\varphi_{s+1}^0\in L^2(\Omega)$ the solution to the adjoint system
  \begin{equation}\label{partiel:syst constant dual rg=0 m forces}
\begin{array}{l}
  \left\{\begin{array}{rcll}
-\partial_t\varphi&=&\Delta \varphi
+\left(\begin{array}{cc}
        \tilde{C}_{11}^*&0\\
        \tilde{C}_{12}^*&\tilde{C}_{22}^*
       \end{array}\right)
\varphi&\mr{~in~}Q_T,\\
        \varphi&=&0&\mr{~on}~\Sigma_T,\\
	\varphi(T)&=&(0,...,0,\varphi^0_{s+1},0,...,0)^{\mr{t}}=e_{s+1}\varphi^0_{s+1} &~\mr{in}~\Omega
        \end{array}
\right.
\end{array}
\end{equation}
 satisfies the observability inequality 
 \begin{equation}\label{partiel:ine obs rg=0 m forces}
    \displaystyle\int_{\Omega}\varphi(0)^2\,dx
  \leqslant C_{obs}\displaystyle\int_{\omega\times(0,T)}(D_1^*(\varphi_1,...,\varphi_s)^{\mr{t}})^2\,dx\,dt.
 \end{equation}
But for all  $\varphi_{s+1}^0\not\equiv0$ in $\Omega$, 
the inequality (\ref{partiel:ine obs rg=0 m forces}) is not satisfied. 
Indeed,  we remark first that, since $\varphi_1(T)=...=\varphi_s(T)=0$ in $\Omega$, 
 we have $\varphi_1=...=\varphi_s=0$ in $Q_T$. 
 Furthermore, if we choose $\varphi_{s+1}^0\not\equiv 0$ in $\Omega$, 
 as previously, 
we get $(\varphi_{s+1}(0),...,\varphi_n(0))\not\equiv0$ in $\Omega$.

  We recall that, as a consequence of Proposition \ref{partiel:property equivalence contr obs}, 
 the $\Pi_p$-null controllability implies the $\Pi_p$-approximate controllability 
of  System  (\ref{partiel:syst contant case r force}). 
If  Condition (\ref{partiel:cond cas constant une force}) is not satisfied, 
as for the $\Pi_p$-null controllability, 
we can find a solution to System (\ref{partiel:syst constant dual rg=0 m forces}) 
such that $D_1^*(\phi_1,...,\phi_s)^{\mr{t}}\equiv 0$ in $\omega\times(0,T)$ 
and $\phi\not\equiv0$ in $Q_T$ and 
we conclude again with Proposition \ref{partiel:property equivalence contr obs}.

\end{proof}

\section{Partial null controllability with time dependent matrices}\label{partiel:section temps}

\hspace*{4mm} We recall  that $[A|B](\cdot)=(B_0(\cdot)|...|B_{n-1}(\cdot))$ (see (\ref{partiel:def B_i})). 
Since 
$A(t)\in \mc{C}^{n-1}([0,T];\mc{L}(\mb{R}^n))$ 
and $ B(t)\in\mc{C}^{n}([0,T];\mc{L}(\mb{R}^m;\mb{R}^n))$, 
we remark that the matrix $[A|B]$ is well defined and is an element of 
$\mc{C}^1([0,T],\mc{L}(\mb{R}^{mn},\mb{R}^n)$. 
We will use the notation $B_i=:(b_1^i|...|b_m^i)$ for all $i\in  \{0,...,n-1\}$. 
To prove Theorem \ref{partiel:theo temps}, we will use the following 
lemma of \cite{gonzalez2010controllability}

\begin{Lemme}\label{partiel:lemme expression base temps}
 Assume that $\mr{max}\{\mr{rank}~[A|B](t):t\in [0,T]\}=s\leqslant n$. 
 Then there exist 
 $T_0,~T_1\in[0,T]$, with $T_0<T_1$, $r\in \{1,...,m\}$ and 
 sequences $(s_j)_{1\leqslant j\leqslant r}\subset \{1,...,n\}$, 
 with $\sum_{i=1}^rs_j=s$, 
 and $(l_j)_{1\leqslant j\leqslant r}\subset \{1,...,m\}$ 
 such that, for every $t\in [T_0,T_1]$, the set 
 \begin{equation}\label{partiel:definition de la base B}
  \mc{B}(t)=
  \bigcup_{j=1}^r\{b_0^{l_j}(t),b_1^{l_j}(t),...,b_{s_j-1}^{l_j}(t)\},
 \end{equation}
is linearly independent, spans the columns of $[A|B](t)$ and satisfies
\begin{equation}\label{partiel:combination lineaire b_i}
 b_{s_j}^{l_j}(t)=
 \sum\limits_{k=1}^j\left(\theta_{s_j,0}^{l_j,l_k}(t)b_0^{l_k}(t)+
\theta_{s_j,1}^{l_j,l_k}(t)b_1^{l_k}(t)+...+
\theta_{s_j,s_k-1}^{l_j,l_k}(t)b_{s_k-1}^{l_k}(t)\right),
 \end{equation}
for every $t\in [T_0,T_1]$ and $j\in\{ 1,..., r\}$, 
where 
\begin{equation*}
 \theta_{s_j,0}^{l_j,l_k}(t), ~
\theta_{s_j,1}^{l_j,l_k}(t),~...,
\theta_{s_j,s_k-1}^{l_j,l_k}(t)\in\mc{C}^1([T_0,T_1]).
\end{equation*}
\end{Lemme}

With exactly the same argument for the proof of the previous lemma, 
we can obtain the
\begin{Lemme}
 If  $\mr{rank~}[A|B](T)=s$, 
 then the conclusions of Lemma \ref{partiel:lemme expression base temps} hold true  
 with $T_1=T$.
\end{Lemme}

\begin{proof}[Proof of Theorem \ref{partiel:theo temps}]
Let $y_0\in L^2(\Omega)^n$ 
 and $s$ be the rank of the matrix $[A|B](T)$. As in the proof 
 of the controllability by one force with constant matrices, 
 let $X$ being the linear space 
 spanned by the columns of the matrix $[A|B](T)$. We consider 
$\mc{B}=\mc{B}(t)$ the basis of $X$ defined in (\ref{partiel:definition de la base B}).

As in the constant case, we will prove that 
we need only $r$ forces to control System 
(\ref{partiel:syst para lin equa contr intro}) that is we study the partial 
null controllability of System (\ref{partiel:syst contant case r force}) with the coupling matrix 
$A(t)\in \mc{C}^{n-1}([0,T];\mc{L}(\mb{R}^n))$ and 
the control matrix $\tilde{B}(t)=(B_{l_1}(t)|B_{l_2}(t)|\cdots|B_{l_r}(t))
\in \mc{C}^{n}([0,T];\mc{L}(\mb{R}^r,\mb{R}^n))$.
If we define $M$ as the matrix whose columns are the elements of 
$\mc{B}(t)$, i.e. for all $t\in[0,T]$
 \begin{equation*}
  M(t)=(m_{ij}(t))_{1\leqslant i\leqslant n, 1\leqslant j\leqslant s}
  :=\left(b^{l_1}_0(t)|b^{l_1}_1(t)|...|b^{l_1}_{s_1-1}(t)|...
  |b^{l_r}_0(t)|b^{l_r}_1(t)|...|b^{l_r}_{s_r-1}(t)\right),
 \end{equation*}
we can remark 
that 
\begin{equation}\label{partiel:rank P11=p p=s temps}
\mr{rank}~\Pi_pM(T)=\mr{rank}~\Pi_p[A|B](T)=p.  
\end{equation}
Indeed, using (\ref{partiel:combination lineaire b_i}), 
\begin{equation*}
\Pi_p b_{s_j}^{l_j}(t)=
 \sum\limits_{k=1}^j\left(\theta_{s_j,0}^{l_j,l_k}(t)\Pi_pb_0^{l_k}(t)+
\theta_{s_j,1}^{l_j,l_k}(t)\Pi_pb_1^{l_k}(t)+...+
\theta_{s_j,s_k-1}^{l_j,l_k}(t)\Pi_pb_{s_k-1}^{l_k}(t)\right).
 \end{equation*}
%

\begin{enumerate}[\textit{Case} 1]
 \item : $p=s$. 
 As in the constant case, we want to apply a change of variable $P$ to the solution $y$ to System (\ref{partiel:syst contant case r force}). 
 Let us define for all $t\in [0,T]$ the following matrix
  \begin{equation}\label{partiel:def P temps p=s}
  P(t):=(b^{l_1}_0(t)|b^{l_1}_1(t)|...|b^{l_1}_{s_1-1}(t)|...
  |b^{l_r}_0(t)|b^{l_r}_1(t)|...|b^{l_r}_{s_r-1}(t)|P_{s+1}(t)|...|P_n(t))\in \mc{L}(\mb{R}^n),
 \end{equation}
 where for all $i\in \{ s+1,...,n\}$, $P_l$ is solution in $\mc{C}^1([0,T])^n$ to the 
 system of ordinary differential equations
 \begin{equation}\label{partiel:system K_2 p=s temps}
  \left\{\begin{array}{l}
          \partial_tP_l(t)=AP_l(t)\mr{~in~}[0,T],\\
          P_l(T)=e_l.
         \end{array}
\right.
 \end{equation}
Using (\ref{partiel:def P temps p=s}) and (\ref{partiel:system K_2 p=s temps}), $P(T)$ can be rewritten
\begin{equation}\label{partiel:definition K(T) p=s temps}
 P(T)=\left(\begin{array}{cc}
        P_{11}&0\\
        P_{21}&I_{n-s}
            \end{array}
\right),
\end{equation}
where 
$P_{11}:=\Pi_p(b^{l_1}_0(T)|b^{l_1}_1(T)|...|b^{l_1}_{s_1-1}(T)|...
  |b^{l_r}_0(T)|b^{l_r}_1(T)|...|b^{l_r}_{s_r-1}(T))\in \mc{L}(\mb{R}^s)$
  and $P_{21}\in\mc{L}(\mb{R}^{n-s},\mb{R}^s)$. 
Using (\ref{partiel:rank P11=p p=s temps}), $P_{11}$, and thus $P(T)$, are invertible. 
Furthermore, since $P$ is continuous on $[0,T]$, 
there exists a $T^*\in [0,T)$  such that 
$P(t)$ is invertible for all $t\in[T^*,T]$. 

\hspace*{4mm} As previously it is sufficient to prove the result for  $T^*=0$. 
Since $P(t)\in \mc{C}^1([0,T],\mc{L}(\mb{R}^n))$  and is invertible on the time interval $[0,T]$, 
again, 
for a fixed 
control $v\in L^2(Q_T)^r$, 
$y$ is the solution to System (\ref{partiel:syst contant case r force}) 
if and only if $w:=P(t)^{-1}y$ is the solution to System 
(\ref{partiel:syst contant case w lemme one force}) 
where  $C$, $D$ are given by 
\begin{equation*}
 C(t):=-P^{-1}(t)\partial_t P(t)+P^{-1}(t)AP(t)\mr{~~~and~~~} D(t):=P^{-1}(t)\tilde{B},
\end{equation*}
for all $t\in[0,T]$.  
Using (\ref{partiel:combination lineaire b_i}) 
and (\ref{partiel:system K_2 p=s temps}), we obtain
\begin{equation}\label{partiel:equation K p<s temps}
 \left\{\begin{array}{rcll}
      -\partial_tP(t)+AP(t)&=&(b^{l_1}_1(t)|b^{l_1}_2(t)|...|b^{l_1}_{s_1}(t)|...
  |b^{l_r}_1(t)|b^{l_r}_2(t)|...|b^{l_r}_{s_r}(t)|0|...|0),\\&=&P(t)\left(
  \begin{array}{cc}
   \tilde{C}_{11}&0\\
   0&0
  \end{array}
  \right)&\mr{~in~}[0,T],\\
      P(t)e_{S_i}&=&b^{l_i}_0&\mr{~in~}[0,T],
        \end{array}
\right.
\end{equation} 
where $S_i=1+\sum_{j=1}^{i-1}s_j$ for $1\leqslant i\leqslant r$, 
  \begin{equation}\label{partiel:def C11 reciproque temps}
  \tilde{C}_{11}:=\left(
  \begin{array}{cccc}
   C_{11}&C_{12}&\cdots&C_{1r}\\
   0&C_{22}&\cdots&C_{2r}\\
    \vdots&\vdots&\ddots&\vdots\\
   0&0&\cdots&C_{rr}\\
  \end{array}
  \right)\in\mc{L}(\mb{R}^s),
  \end{equation}
and for $1\leqslant i\leqslant j\leqslant r$, 
the matrices 
 $C_{ij}\in\mc{C}^{0}([0,T];\in\mc{L}(\mb{R}^{s_j},\mb{R}^{s_i}))$  
 are given here by
 \begin{equation}\label{partiel:def Cij temps}
 C_{ii}=\left(
  \begin{array}{ccccc}
  0&0&0&\ldots&\theta_{s_i,0}^{l_i,l_i}\\
  1&0&0&\ldots&\theta_{s_i,1}^{l_i,l_i}\\
  0&1&0&\ldots&\theta_{s_i,2}^{l_i,l_i}\\
  \vdots&\vdots&\ddots&\ddots&\vdots\\
  0&0&\ldots&1&\theta_{s_i,s_i-1}^{l_i,l_i}
  \end{array}
  \right) 
   ~~\mr{~~and~}
C_{ij}=\left(
  \begin{array}{ccccc}
  0&0&0&\ldots&\theta_{s_j,0}^{l_j,l_i}\\
  0&0&0&\ldots&\theta_{s_j,1}^{l_j,l_i}\\
  0&0&0&\ldots&\theta_{s_j,2}^{l_j,l_i}\\
  \vdots&\vdots&\ddots&\ddots&\vdots\\
  0&0&\ldots&0&\theta_{s_j,s_i-1}^{l_j,l_i}
  \end{array}
  \right)\mr{~for~j>i}. 
\end{equation} 
Then 
\begin{equation}\label{partiel:def C p=s temps}
 C=\left(
  \begin{array}{cc}
   \tilde{C}_{11}&0\\
   0&0
  \end{array}
  \right)\mr{~and~}D=(e_{S_1}|...|e_{S_r}).
\end{equation}
Using Theorem \ref{partiel:theo cascade}, there exists 
$v\in L^2(Q_T)^r$ such that the solution to System (\ref{partiel:syst contant case w lemme one force}) 
satisfies  $w_1(T)=...=w_s(T)\equiv0$ in $\Omega$. Moreover, 
the equality (\ref{partiel:definition K(T) p=s temps}) leads to 
 \begin{equation*}
\Pi_sy(T)=(y_{1}(T),...,y_{s}(T))^{\mr{t}}
=P_{11}(w_1(T),...,w_s(T))^{\mr{t}}\equiv0\mr{~in~}\Omega. 
\end{equation*}

 \item : $p<s$. 
 The same method as in the constant case leads to the conclusion (see $\S$ \ref{partiel:section one force}). 
\end{enumerate}

The $\pi_p$-approximate controllability can proved also as in the constant case.

\end{proof}

\section{Partial null controllability for a space dependent coupling matrix}

 \hspace*{4mm}
All along this section, the dimension $N$ will be equal to $1$, more precisely $\Omega:=(0,\pi)$ 
with the exception of the proof of the third point in Theorem \ref{partiel:theo example contr intro} and the 
numerical illustration in Section \ref{partiel:section num} where $\Omega:=(0,2\pi)$. 
We recall that the eigenvalues of $-\Delta$ in $\Omega$ with Dirichlet boundary conditions 
are given by $\mu_k:=k^2$ for all $k\geqslant1$ and we will denote by 
 $(w_k)_{k\geqslant1}$ the associated $L^2$-normalized eigenfunctions. 
 Let us consider the following parabolic system of two equations
 \begin{equation}\label{partiel:syst prim}
\begin{array}{l}
  \left\{\begin{array}{ll}
   \partial_t y=\Delta y+\alpha z+\mathds{1}_{\omega}u&\mr{~in~}Q_T,\\
    \partial_t z=\Delta z&\mr{~in~}Q_T,\\
        y=z=0&\mr{~on}~\Sigma_T,\\
	y(0)=y_0,~z(0)=z_0 &~\mr{in}~\Omega,
        \end{array}
\right.
\end{array}
\end{equation}
where $y_0,~z_0\in L^2(\Omega)$ are the initial data, $u\in L^2(Q_T)$ is the control 
and the coupling coefficient $\alpha$ is in $ L^{\infty}(\Omega)$. 
 We recall that System (\ref{partiel:syst prim}) is $\Pi_1$\textit{-null controllable} if for all 
 $y^0,z^0\in L^2(\Omega)$, we can find a control $u\in L^2(Q_T)$ such that the solution 
 $(y,z)\in W(0,T)^2$ 
 to System (\ref{partiel:syst prim}) satisfies
 $y(T)\equiv 0$ in $\Omega$.

 \subsection{Example of controllability}\label{partiel:example control}


 \hspace*{4mm} In this subsection, we will provide an example of $\Pi_1$-null controllability for 
 System (\ref{partiel:syst prim}) with the help of the method of moments 
 initially developed in  \cite{fatorrinirussell71}. 
 As already mentioned, we suppose that $\Omega:=(0,\pi)$, 
 but the argument of Section \ref{partiel:example control} can be adapted for any open bounded interval of $\mb{R}$. 
 Let us introduce the adjoint system associated to our control problem
 \begin{equation}\label{partiel:syst dual}
\begin{array}{l}
  \left\{\begin{array}{ll}
   -\partial_t \phi=\Delta \phi&\mr{~in~}(0,\pi)\times(0,T),\\
   - \partial_t \psi=\Delta \psi+\alpha\phi&\mr{~in~}(0,\pi)\times(0,T),\\
        \phi(0)=\phi(\pi)=\psi(0)=\psi(\pi)=0&\mr{~on}~(0,T),\\
	\phi(T)=\phi_0,~\psi(T)=0 &~\mr{in}~(0,\pi),
        \end{array}
\right.
\end{array}
\end{equation}
where $\phi_0\in L^2(0,\pi)$. 
For an initial data $\phi_0\in L^2(0,\pi)$ in adjoint System (\ref{partiel:syst dual}), we get 
\begin{equation}\label{partiel:egualite dual0}
\displaystyle\int_0^{\pi}\phi_0y(T)\,dx-\displaystyle\int_0^{\pi}\phi(0)y_0\,dx
- \displaystyle\int_0^{\pi}\psi(0)z_0\,dx
=\displaystyle\iint_{q_T}\phi u\,dx\,dt,
\end{equation}
with the notation $q_T:=\omega\times(0,T)$. 
Since $(w_k)_{k\geqslant1}$ spans $L^2(0,\pi)$, 
System (\ref{partiel:syst prim}) is $\Pi_1$\textit{-null controllable} if and only if 
there exists $u\in L^2(q_T)$ such that, for all  $k\in\mb{N}^*$, 
the solution to System (\ref{partiel:syst dual})  satisfies the following equality
\begin{equation}\label{partiel:egualite dual}
-\displaystyle\int_0^{\pi}\phi_k(0)y_0\,dx- \displaystyle\int_0^{\pi}\psi_k(0)z_0\,dx
=\displaystyle\iint_{q_T}\phi_k u\,dx\,dt,
\end{equation}
where $(\phi_k,\psi_k)$ is the solution to adjoint System (\ref{partiel:syst dual}) for the initial 
data $\phi_0:=w_k$.
 
 
Let $k\in\mb{N}^*$. With the initial condition $\phi_0:=w_k$ is associated the solution 
 $(\phi_k,\psi_k)$  to adjoint System (\ref{partiel:syst dual}):
\begin{equation*}
 \phi_k(t)=
 e^{-k^2(T-t)}
 w_k\mr{~in~}(0,\pi)
\end{equation*}
for all $t\in[0,T]$. If we write:
\begin{equation*}
\psi_k(x,t):=\sum\limits_{l\geqslant 1}
\psi_{kl}(t)w_l(x)\mr{~for~all~}(x,t)\in (0,\pi)\times(0,T),  
\end{equation*}
then a simple computation leads to the formula
\begin{equation}\label{partiel:expres psi cont}
 \begin{array}{rcl}
  \psi_{kl}(t)
  &=&
  \dfrac{e^{-k^2(T-t)}-e^{-l^2(T-t)}}{-k^2+l^2}\alpha_{kl}
  \mr{~for~all~}l\geqslant 1,~t\in (0,T),
 \end{array}
\end{equation}
where,  for all $k,l\in \mb{N}^*$, $\alpha_{kl}$ is defined in \eqref{partiel:def alpha kj}.
In (\ref{partiel:expres psi cont}) we implicitly used  the convention: if $l=k$ the term  
$ (e^{-k^2(T-t)}-e^{-l^2(T-t)})/(-k^2+l^2)$ 
is replaced by $(T-t)e^{-k^2(T-t)}$. 
With these expressions of $\phi_k$ and $\psi_k$, 
the equality (\ref{partiel:egualite dual}) reads for all $k\geqslant 1$
\begin{equation}\label{partiel:ine obs int partiel}
-e^{-k^2T}y_k^0
- \sum\limits_{l\geqslant1} \dfrac{e^{-k^2T}-e^{-l^2T}}{-k^2+l^2}
 \alpha_{kl}z_l^0
 =\displaystyle\iint_{q_T}e^{-k^2(T-t)}w_k(x)u(t,x)\,dx\,dt.
\end{equation}


%
%

In the proof of Theorem \ref{partiel:theo example contr intro}, 
we will look for a control $u$ expressed as $u(x,t)=f(x)\gamma(t)$ 
with $\gamma(t)=\sum_{k\geqslant 1}\gamma_kq_k(t)$ and  
$(q_k)_{k\geqslant 1}$ a family biorthogonal to $(e^{-k^2t})_{k\geqslant 1}$.  
  Thus, we will need the two following lemma
\begin{Lemme}\label{partiel:lemme contr separe} (see  Lemma 5.1, \cite{Ammar-Khodja2015}) 
There exists $f\in L^2(0,\pi)$ such that $\mr{Supp} ~f\subset \omega$ and for a constant $\beta$, one has
 $$\inf\limits_{k\geqslant1}f_kk^3=\beta>0,$$ 
 where, for all $k\in\mb{N}^*$, $f_k:=\int_{0}^{\pi}fw_k\,dx$.
\end{Lemme}

\begin{Lemme}\label{partiel:lemme qk} (see Corollary 3.2, \cite{fatorrinirussell71}) 
 There exists a  sequence $(q_k)_{k\geqslant 1}\subset L^2(0,T)$ 
 biorthogonal to $(e^{-k^2t})_{k\geqslant 1}$, that is
 \begin{equation*}
  \langle q_k,e^{-l^2t}\rangle_{L^2(0,T)}=\delta_{kl}.
 \end{equation*}
Moreover, for all $\varepsilon>0$, there exists $C_{T,\varepsilon}>0$, independent of $k$, such that
\begin{equation}\label{partiel:lemme qk inegalite}
\|q_k\|_{L^2(0,T)}\leqslant C_{T,\varepsilon}e^{(\pi+\varepsilon) k},~\forall k\geqslant 1.
\end{equation}
\end{Lemme}

\begin{rem}
When $\Omega:=(a,b)$ with $a,b\in\mb{R}$, the inequality \eqref{partiel:lemme qk inegalite} of Lemma \ref{partiel:lemme qk} is replaced by 
\begin{equation*}
\|q_k\|_{L^2(0,T)}\leqslant C_{T,\varepsilon}e^{(b-a+\varepsilon) k},~\forall k\geqslant 1.
\end{equation*}
\end{rem}

\begin{proof}[Proof of the second point in Theorem \ref{partiel:theo example contr intro}]
 As mentioned above, let us look for the control $u$ of the form $u(x,t)=f(x)\gamma(t)$, 
 where $f$ is as in Lemma 
\ref{partiel:lemme contr separe}. 
Since $ f_k\neq0$ for all $k\in\mb{N}^*$, using (\ref{partiel:ine obs int partiel}), 
the $\Pi_1$-null controllability of System 
 (\ref{partiel:syst prim}) is reduced  to find a solution $\gamma\in L^2(0,T)$  
 to the following problem of moments:
 \begin{equation}\label{partiel:pbl moment}
  \displaystyle\int_0^T\gamma(T-t)e^{-k^2t}\,dt
  =f_k^{-1}\left(-e^{-k^2T}y_k^0
- \sum\limits_{l\geqslant1} \dfrac{e^{-k^2T}-e^{-l^2T}}{-k^2+l^2}
 \alpha_{kl}z_l^0
  \right):=M_k~\forall k\geqslant 0.
 \end{equation}
 The function $\gamma(t):=\sum_{k\geqslant1}M_kq_k(T-t)$ is a solution to this problem of moments. 
 We need only to prove  that $\gamma\in L^2(0,T)$.  
Using the convexity of the exponential function, we get for all $k\in \mb{N}^*$,
\begin{equation}\label{partiel:ine moment 1}
\begin{array}{rcl}
 \sum\limits_{l\geqslant1}\left|\dfrac{e^{-k^2T}-e^{-l^2T}}{-k^2+l^2}\right|
 |\alpha_{kl}|
 &=& \sum\limits_{l=1}^k\left|\dfrac{e^{-k^2T}-e^{-l^2T}}{-k^2+l^2}\right| |\alpha_{kl}|
 +\sum\limits_{l=k+1}^{\infty}\left|\dfrac{e^{-k^2T}-e^{-l^2T}}{-k^2+l^2}\right|
 |\alpha_{kl}|\\
 &\leqslant & \sum\limits_{l=1}^kTe^{-l^2T} |\alpha_{kl}|
 +\sum\limits_{l=k+1}^{\infty}Te^{-k^2T} |\alpha_{kl}|\\
 &=:&A_{1,k}+A_{2,k}.
\end{array}\end{equation}
With the Condition (\ref{partiel:second point t space}) on $\alpha$, there exists a positive constant $C_T$ which do not depend on $k$ such that 
for all $k\in\mb{N}^*$
\begin{equation}\label{partiel:ine moment 2}\begin{array}{rcl}
 A_{1,k}\leqslant C_1T\sum\limits_{l=1}^ke^{-l^2T} e^{-C_2(k-l)}
 \leqslant C_1Te^{-C_2k}\sum\limits_{l=1}^{\infty} e^{-l^2T+C_2 l}
 \leqslant C_Te^{-C_2k}
 \end{array}\end{equation}
and
 \begin{equation}\label{partiel:ine moment 3}\begin{array}{rcl}
  A_{2,k}\leqslant C_1Te^{-k^2T}\sum\limits_{l=k+1}^{\infty} e^{-C_2(l-k)}
   \leqslant C_1Te^{-k^2T}\sum\limits_{j=0}^{\infty} (e^{-C_2})^j
  \leqslant C_1Te^{-k^2T}\dfrac{1}{1-e^{-C_2}}.
 \end{array}
 \end{equation}
Combining the three last inequalities 
\eqref{partiel:ine moment 1}-\eqref{partiel:ine moment 3}, for all $k\in\mb{N}^*$
 \begin{equation}\label{partiel:preuve lemme cont}
  \begin{array}{rcl}
 \sum\limits_{l\geqslant1}\left|\dfrac{e^{-k^2T}-e^{-l^2T}}{-k^2+l^2}\right|
 |\alpha_{kl}|
 &\leqslant& C_{T}e^{-C_2k},
\end{array}
 \end{equation}
where  $C_T$ is a positive constant  independent of $k$. 
Let $\varepsilon\in (0,1)$. Then, with Lemma \ref{partiel:lemme contr separe},  \eqref{partiel:pbl moment} and \eqref{partiel:preuve lemme cont}, 
there exists a positive constant $C_{T,\varepsilon}$ independent of $k$ such that for all $k\in\mb{N}^*$
\begin{equation*}\begin{array}{rcl}
 |M_k|&\leqslant &\beta^{-1}k^3
 \left(e^{-k^2T}\|y_0\|_{L^2(0,\pi)}+C_Te^{-C_2k}\|z_0\|_{L^2(0,\pi)}
 \right)\\
 &\leqslant &C_{T,\varepsilon}e^{-C_2(1-\varepsilon)k}(\|y_0\|_{L^2(0,\pi)}+\|z_0\|_{L^2(0,\pi)}).
\end{array}\end{equation*}
Thus, using Lemma \ref{partiel:lemme qk}, for $\varepsilon$ small enough and a positive constant $C_{T,\varepsilon}$ 
\begin{equation*}
 \|\gamma\|_{L^2(0,T)}
 \leqslant 
C_{T,\varepsilon}(\sum\limits_{k\in \mb{N}^*}e^{-[C_2(1-\varepsilon)-\pi+\varepsilon]k})
(\|y_0\|_{L^2(0,\pi)}+\|z_0\|_{L^2(0,\pi)})<\infty.
\end{equation*}
 \end{proof}

\subsection{Example of non controllability}\label{partiel:section contre exemple}
\hspace*{4mm} In this subsection, to provide an example of non 
$\Pi_1$-null controllability of System (\ref{partiel:syst prim}), 
we will first study the boundary controllability of the following parabolic system 
 of two equations
 \begin{equation}\label{partiel:syst prim bord}
\begin{array}{l}
  \left\{\begin{array}{ll}
   \partial_t y=\Delta y+\alpha z&\mr{~in~}Q_T:=(0,\pi)\times(0,T),\\
    \partial_t z=\Delta z&\mr{~in~}Q_T,\\
        y(0,t)=v(t), ~y(\pi,t)=z(0,t)=z(\pi,t)=0&\mr{~on}~(0,T),\\
	y(x,0)=y_0(x),~z(x,0)=z_0(x) &~\mr{in}~\Omega:=(0,\pi),
        \end{array}
\right.
\end{array}
\end{equation}
where $y_0,~z_0\in H^{-1}(0,\pi)$ are the initial data, $v\in L^2(0,T)$ is the boundary control 
and $\alpha\in L^{\infty}(0,\pi)$. 
For any given $y_0,z_0\in H^{-1}(0,\pi)$ and $v\in L^2(0,T)$, 
System (\ref{partiel:syst prim bord}) has a unique solution in 
$L^2(Q_T)^2\cap\mc{C}^0([0,T];H^{-1}(\Omega)^2)$ (defined by transposition; 
see \cite{fernandezcaraboundary2010}). 

As in Section \ref{partiel:example control}, 
for an initial data $(y_0,z_0)\in H^{-1}(0,\pi)^2$ we can find a control $v\in L^2(0,T)$ such that the solution 
to (\ref{partiel:syst prim bord}) satisfies $y(T)\equiv0$ in $(0,\pi)$ if and only if   for all  $\phi_0 \in H^1_0(0,\pi)$
the solution to System (\ref{partiel:syst dual})  verifies the equality
  \begin{equation}\label{partiel:egalite non cont dualite bis} 
 - \langle y_0,\phi(0)\rangle_{H^{-1},H^1_0}
  -\langle z_0,\psi(0)\rangle_{H^{-1},H^1_0}
  =\displaystyle\int_0^Tv(t)\phi_x(0,t)\,dt,
 \end{equation}
where the duality bracket $\langle\cdot,\cdot\rangle_{H^{-1},H^1_0}$ is defined as  $\langle f,g\rangle_{H^{-1},H^1_0}:=f(g)$  for all $f\in H^{-1}(0,\pi)$ and all $g\in H^1_0(0,\pi)$. 

The used strategy here is inspired from \cite{GerreroImanuvilovMem13}. The idea involves constructing  particular initial data for adjoint System (\ref{partiel:syst dual}):

\begin{Lemme} \label{partiel:lemme non cont}
Let  $m,G\in \mb{N}^*$. For all $M\in\mb{N}\backslash\{0,1\}$, there exists $\phi_{0,M}\in L^2(0,\pi)$ given 
by $$\phi_{0,M}=\sum\limits_{i=1}^m\phi_{GM+i}^{0,M}w_{GM+i},$$
with $\phi_{GM+1}^{0,M},...,\phi_{GM+m}^{0,M}\in \mb{R}$, such that  the solution $(\phi_M,\psi_M)$ to adjoint System (\ref{partiel:syst dual}) with $\phi_0=\phi_{0,M}$ satisfies
\begin{equation}\label{partiel:estim B}
\begin{array}{rcl}
 \left(\displaystyle\int_0^T(\phi_{M})_x(0,t)^2\,dt\right)^{1/2}   &\leqslant&\dfrac{\gamma_1}{M^{(2m-5)/2}},
 \end{array}
 \end{equation}
 where $\gamma_1$ does not depend on $M$. Morover  for an increasing sequence $(M_j)_{j\in\mb{N}}\subset\mb{N}\backslash\{0,1\}$ 
 and a  $k_1\in\{1,...,m\}$, we have $|\phi_{ GM_j+k_1}^{0,j}|=1$ for all $G\in \mb{N}^*$ and $j\in\mb{N}$. 
\end{Lemme}
To study the controllability of System \eqref{partiel:syst prim bord} we will use the fact that for fixed $m,G\in\mb{N}^*$, 
the quantity in the left-side hand in \eqref{partiel:estim B} converge to zero when $M$ goes to infinity.

\begin{proof}We remark first that 
\begin{equation}\label{partiel:ine obs d}
A_M:=\displaystyle\int_0^T(\phi_M)_x(0,t)^2\,dt =\displaystyle\int_0^T\left| \sum\limits_{k=GM+1}^{GM+m}
 ke^{-k^2(T-t)}\phi^{0,M}_k\right|^2\,dt.
\end{equation}
We can rewrite $A_M$ as follows:
\begin{equation}\label{partiel:deuxieme egalite B}
 A_M= \displaystyle\int_0^T\left| \sum\limits_{j=1}^m
 (GM+j)e^{-(G^2M^2+2GMj+j^2)(T-t)}\phi^{0,M}_{GM+j}\right|^2\,dt
= \displaystyle\int_0^Te^{-2 G^2M^2(T-t)}g_{M}(t)\,dt,
\end{equation}
where, for all $t\in [0,T]$, $g_{M}(t):=f_{M}(t)^2$ with
\begin{equation*}
 f_{M}(t):=\sum\limits_{j=1}^{m}(GM+j)e^{-(2GMj+j^2)(T-t)}\phi^{0,M}_{GM+j}.
\end{equation*}
Let $(\phi^{0,M}_{GM+1},\phi_{GM+2}^{0,M},...,\phi_{GM+m}^{0,M})$ be 
a  nontrivial solution of the following homogeneous linear system 
of $m-1$ equations with $m$ unknowns
\begin{equation}\label{partiel:system f_m(0)}
 f_{M}^{(l)}(T)=\sum\limits_{j=1}^{m}(GM+j)(2GMj+j^2)^l\phi_{GM+j}^{0,M}=0
 ,\mr{~for~all~}~l\in\{0,...,m-2\}.
\end{equation}
Using Leibniz formula
\begin{equation*}
 g^{(l)}_M=\sum\limits_{k=0}^l\left(\begin{array}{c}l\\k\end{array}\right)
 f_M^{(k)}f_M^{(l-k)}
\end{equation*}
we deduce that
\begin{equation}\label{partiel:cond g_M}
 g_{M}^{(l)}(T)=0,\mr{~for~all~}~l\in\{0,...,2m-4\}.
\end{equation}
Using (\ref{partiel:cond g_M}), after $2m-3$ integrations by part in (\ref{partiel:deuxieme egalite B}), we obtain
\begin{equation*}\begin{array}{rcl}
 A_M&=& \dfrac{-g_{M}(0)e^{-2G^2M^2T}}{2G^2M^2}
+\displaystyle\int_0^T\dfrac{e^{-2G^2M^2(T-t)}}{(-2 G^2M^2)}g_{M}^{(1)}(t)\mr{d}t\vspace*{5mm}\\
 &=&\sum\limits_{l=0}^{2m-4}\dfrac{g_{M}^{(l)}(0)e^{-2G^2M^2T}}{(-2G^2M^2)^{l+1}}
+\displaystyle\int_0^T\dfrac{e^{-2G^2M^2(T-t)}}{(-2G^2M^2)^{2m-3}}g_{M}^{(2m-3)}(t)\,dt.
 \end{array}\end{equation*}
By linearity, in (\ref{partiel:system f_m(0)}) we can choose 
$\phi_{GM+1}^{0,M},~...,\phi_{GM+m}^{0,M}$ such that 
\begin{equation}\label{partiel:non contr borne donnnee init}
 \sup\limits_{i\in\{1,...,m\}}|\phi^{0,M}_{GM+i}|=1.
\end{equation}
Thus, 
for  all $l\in \mb{N}$ and all $t\in [0,T]$, 
the following estimate holds
\begin{equation*}\begin{array}{rcl}
 |g^{(l)}_M(t)|&=&\left|\sum\limits_{k=0}^l\left(\begin{array}{c}l\\k\end{array}\right)
 f_M^{(k)}(t)f_M^{(l-k)}(t)\right|\\
 &\leqslant&\sum\limits_{k=0}^l\left(\begin{array}{c}l\\k\end{array}\right)
\left| \sum\limits_{j=1}^{m}(GM+j)(2GMj+j^2)^ke^{-(2GMj+j^2)(T-t)}\phi^{0,M}_{GM+j}\right|\\
&&~~~~~~~~~~~~~~~~~~~~~~~~~~~~
\times\left| \sum\limits_{j=1}^{m}(GM+j)(2GMj+j^2)^{l-k}e^{-(2GMj+j^2)(T-t)}\phi^{0,M}_{GM+j}\right|\\
&\leqslant&(GM+m)^2m^2\sum\limits_{k=0}^l\left(\begin{array}{c}l\\k\end{array}\right)
(2GMm+m^2)^l\\
&\leqslant&CM^{l+2},
\end{array}\end{equation*}
where $C$ does not depend on $M$. 
Then, since  $\sup\limits_{i\in\{1,...,m\}}|\phi^{0,M}_{GM+i}|=1$, there exist $C,\tau>0$ such that 
\begin{equation*}
\begin{array}{rcl}
 A_M&\leqslant & 
 e^{-2G^2M^2T}\sum\limits_{l=0}^{2m-4}\dfrac{\|g_{M}^{(l)}\|_{\infty}}{(2G^2M^2)^{l+1}}
+\dfrac{T\|g_{M}^{(2m-3)}\|_{\infty}}{(2G^2M^2)^{2m-3}}\\
&\leqslant&
 e^{-\tau M^2}\sum\limits_{l=0}^{\infty}\dfrac{C}{M^{l}}+\dfrac{C}{M^{2m-5}}\\
 &\leqslant& CM^{-2}e^{-\tau M^2}\dfrac{1}{1-M^{-2}}+\dfrac{C}{M^{2m-5}}.
 \end{array}
 \end{equation*}
Thus there exists $\gamma_1>0$ such that 
we have the estimate
\begin{equation*}
\begin{array}{rcl}
 A_M &\leqslant&\dfrac{\gamma_1}{M^{2m-5}},
 \end{array}
 \end{equation*}
where $\gamma_1$ does not depend on $M$. 
Using (\ref{partiel:phi0 GM+k_1}), for all $M\geqslant 2$, 
 there exists $k_1(M)\in\{1,...,7\}$, such that 
 $|\phi_{15M+k_1(M)}^{0,M}|=1$.
Thus there exists an increasing sequence $(M_j)_{j\in\mb{N}^*}$ such that   $|\phi_{15M_j+k_1}^{0,M_j}|=1$ for a $k_1\in\{1,...,m\}$ independent of $j$.
\end{proof}

\begin{theo}\label{partiel:theo contr exemple bord}
Let $T>0$ and  $\alpha$ be the function of $L^{\infty}(0,\pi)$ defined by
\begin{equation}\label{partiel:def alpha bord}
\alpha(x):=\sum\limits_{j=1}^{\infty}\dfrac{1}{j^2}\cos(15jx)\mr{~for~all~}x\in(0,\pi).
\end{equation}
Then there exists $k_1\in\{1,..,7\}$ such that for $(y_0,z_0):=(0,w_{k_1})$ and all control $v\in L^2(0,T)$, 
the solution to System (\ref{partiel:syst prim bord}) verifies $y(T)\not\equiv0$ in $(0,\pi)$.
\end{theo}

\begin{proof} 
To understand why the number «15» appears in the definition (\ref{partiel:def alpha bord}) 
of the function $\alpha$, we will consider for all $x\in (0, \pi)$
\begin{equation}\label{partiel:expr alpha non cont preuve}
\alpha(x):=\sum\limits_{j=1}^{\infty}\dfrac{1}{j^2}\cos(Gjx)\mr{~for~all~}x\in(0,\pi),
\end{equation}
where $G\in \mb{N}^*$. 
We recall that for an  initial condition $(y_0,z_0)\in L^2(0,\pi)^2$ 
and a control $v\in L^2(0,T)$, the solution to  System (\ref{partiel:def alpha bord}) satisfies $y(T)\equiv0$ in $(0,\pi)$ if and only if 
for all $\phi_0\in L^2(0,\pi)$, we have the equality
  \begin{equation}\label{partiel:egalite non cont dualite} 
 - \langle y_0,\phi(0)\rangle_{H^{-1},H^1_0}
  -\langle z_0,\psi(0)\rangle_{H^{-1},H^1_0}
  =\displaystyle\int_0^Tv(t)\phi_x(0,t)~\,dt,
 \end{equation}
 where $(\phi,\psi)$ is the solution to the adjoint System (\ref{partiel:syst dual}). 
 Let us consider the sequences $(M_j)_{j\in\mb{N}^*}$ and $(\phi_{0,M_j})_{j\in \mb{N}}$, $k_1$ defined 
 in Lemma \ref{partiel:lemme non cont} and $(\phi_{M_j},\psi_{M_j})$ the solution to
 \begin{equation*}
  \left\{\begin{array}{ll}
          -\partial_t\phi_{M_j}=\Delta\phi_{M_j}&\mr{in}~(0,\pi)\times(0,T),\\
          -\partial_t\psi_{M_j}=\Delta\psi_{M_j}+\alpha\phi_{M_j}&\mr{in}~(0,\pi)\times(0,T),\\
          \phi_{M_j}(0)=\phi_{M_j}(\pi)=\psi_{M_j}(0)=\psi_{M_j}(\pi)=0&\mr{on}~(0,T),\\
          \phi_{M_j}(T)=\phi_{0,{M_j}},~\psi_{M_j}(T)=0&\mr{in}~(0,\pi).
         \end{array}
\right.
 \end{equation*}
 The goal is to prove that for the initial data $(y_0,z_0):=(0,w_{k_1})$ and $\phi_{0,{M_j}}$ for $j$ large enough, 
the equality \eqref{partiel:egalite non cont dualite} does not holds.   
 Using Lemma \ref{partiel:lemme non cont}, we have
  \begin{equation}\label{partiel:estim right non cont} 
 \left|\displaystyle\int_0^Tv(t)(\phi_{M_j})_x(0,t)~\,dt\right|
 \leqslant \dfrac{\gamma_1\|v\|_{L^2(q_T)}}{{M_j}^{(2m-5)/2}}.
 \end{equation}
 Since $y_0=0$, we obtain
 \begin{equation}\label{partiel:contradic equalite dualite non cont 1}
\langle y_0,\phi_{M_j}(0)\rangle_{H^{-1},H^1_0}=0.  
 \end{equation}
Let us now estimate the term $\langle z_0,\psi_{M_j}(0)\rangle_{H^{-1},H^1_0}$ in the equality \eqref{partiel:egalite non cont dualite}. 
We recall that the expression of $\alpha$ is given in \eqref{partiel:expr alpha non cont preuve}. 
Then, the function $\alpha$ is of the form $\alpha(x)=\sum\limits_{p=0}^{\infty}\alpha_p\cos(px)$ 
for all $x\in(0,\pi)$, with 
\begin{equation}\label{partiel:alpha_p alpha non cont}
 \alpha_p:=\left\{\begin{array}{ll}
                  \frac{1}{i^2}&~\mr{if~}p=Gi\mr{~with~}i\in\mb{N}^*,\\
                  0&~\mr{otherwise}.
                 \end{array}
                 \right.
\end{equation}
From the definition of $\alpha_{kl}$ in (\ref{partiel:def alpha kj}), 
 there holds for all $k,l\in \mb{N}^*$
  \begin{equation*}
 \begin{array}{rcl}
  \alpha_{kl}  &=&\frac{1}{\pi}(\alpha_{|k-l|}-\alpha_{k+l}).
\end{array} 
\end{equation*}
Let $k\in\{1,...,m\}$ and  $l\in\{ G{M_j}+1,...,G{M_j}+m\}$. 
We have $k+l\in\{ G{M_j}+2,...,G{M_j}+2m\}$. Thus if we choose 
\begin{equation}\label{partiel:G geqslant 2m+1}
 G\geqslant 2m+1,
\end{equation}
using (\ref{partiel:alpha_p alpha non cont}), we obtain
 \begin{equation*}
\alpha_{k+l}=0  
 \end{equation*}
 and 
\begin{equation*}
\alpha_{|k-l|}=\left\{\begin{array}{cl}
                     \dfrac{1}{{M_j}^2}&\mr{~if~}|k-l|=G{M_j},\\
                     0&\mr{~otherwise.}
                    \end{array}
 \right.
\end{equation*}
So that we have the following submatrix of $(\alpha_{kl})_{1\leqslant k,l\leqslant GM+m}$:
\begin{equation}\label{partiel:calcul submat alpha}
 (\alpha_{kl})_{1\leqslant k\leqslant m,G{M_j}+1\leqslant l\leqslant G{M_j}+m}
 =\dfrac{1}{\pi{M_j}^2}I_{\mb{R}^m}.
\end{equation}
According to Lemma \ref{partiel:lemme non cont}, 
there exists $k_1\in\{1,...,m\}$ such that 
\begin{equation}\label{partiel:phi0 GM+k_1}
|\phi^{0,M_j}_{G{M_j}+k_1}|=1.  
\end{equation}
Furthermore, since $k_1\in\{1,...,m\}$,
\begin{equation}\label{partiel:estim k_1}
 |e^{-k_1^2T}-e^{-(G{M_j}+k_1)^2T}|\geqslant |e^{-m^2T}-e^{-G^2{M_j}^2T}|
\end{equation}
and
\begin{equation}\label{partiel:estim k_1 bis}
 (G{M_j}+k_1)^2-k_1^2\leqslant (G{M_j}+m)^2-1.
\end{equation}
 Since  $z_0=w_{k_1}$, 
 the equality  (\ref{partiel:calcul submat alpha}) leads to 
\begin{equation*}\begin{array}{rcl}
  \left|\displaystyle\int_0^{\pi}z_0\psi_{M_j}(0)\,dx\right|
  &=& \left|\sum\limits_{s=1}^7
  \dfrac{e^{-k_1^2T}-e^{-(G{M_j}+s) 2T}}{-k_1^2+(G{M_j}+s)^2}
 \alpha_{k_1,G{M_j}+s}\phi_{G{M_j}+s}^{0,{M_j}}\right|\vspace*{4mm}\\
  &=& \left|\dfrac{e^{-k_1^2T}-e^{-(G{M_j}+k_1)^2T}}{-k_1^2+(G{M_j}+k_1)^2}
 \dfrac{1}{\pi{M_j}^2}\right|.
 \end{array}\end{equation*}
  Then using \eqref{partiel:estim k_1}  and \eqref{partiel:estim k_1 bis} for all $j\in\mb{N}^*$
\begin{equation}\label{partiel:contradic equalite dualite non cont 2}
 \left|\langle z_0,\psi_{M_j}(0)\rangle_{H^{-1},H^1_0}\right|
 =\left|\displaystyle\int_0^{\pi}z_0\psi_{M_j}(0)\,dx\right|
 \geqslant \dfrac{\gamma_2}{{M_j}^4},
 \end{equation}
 where $\gamma_2$ does not depend on $j$. 
 Combining (\ref{partiel:estim right non cont}) and (\ref{partiel:contradic equalite dualite non cont 2}), 
 we obtain a contradiction with equality (\ref{partiel:egalite non cont dualite}). 
 Thus, for this initial condition $y_0$ and $z_0$, 
 we can not find a control $v\in L^ 2(0,T)$ such that the  solution $(y,z)$ to  
 system  (\ref{partiel:def alpha bord}) satisfies $y(T)\equiv0$ in $(0,\pi)$.

\end{proof}

\begin{proof}[Proof of the third point in Theorem \ref{partiel:theo example contr intro}]
 Using  Theorem \ref{partiel:theo contr exemple bord}, 
for the initial data $(p_0,q_0):=(0,w_{k_1})\in L^2(0,\pi)^2$ and all control $v\in L^2(0,T)$, the solution $(p,q)\in W(0,T)^2$ (defined by transposition) to the system
 \begin{equation}\label{partiel:syst prim bord preuve int}
\begin{array}{l}
  \left\{\begin{array}{ll}
   \partial_t p=\Delta p+\alpha q&\mr{~in~}(0,\pi)\times(0,T),\\
    \partial_t q=\Delta q&\mr{~in~}(0,\pi)\times(0,T),\\
        p(\pi,\cdot)=v, ~p(0,\cdot)=q(0,\cdot)=q(\pi,\cdot)=0&\mr{~on}~(0,T),\\
	p(\cdot,0)=p_0,~q(\cdot,0)=q_0 &~\mr{in}~(0,\pi)
        \end{array}
\right.
\end{array}
\end{equation}
satisfies $p(T)\not\equiv0$ in $(0,\pi)$. 
Consider now $\overline{p}_0$, $\overline{q}_0\in L^2(0,2\pi)$ defined by
\begin{equation*}
\overline{p}_0(x)=0\mr{~~~and~~~}
 \overline{q}_0(x)=\sqrt{\dfrac{2}{\pi}}\sin(k_1x)\mr{~~~for~all~}x\in (0,2\pi).
\end{equation*}
Remark that $(\overline{p}_{0|(0,\pi)},\overline{q}_{0|(0,\pi)})=(p_0,q_0)$. 
Let $\omega\subset(0,\pi)$. 
Suppose now that the system 
 \begin{equation}\label{partiel:syst prim preuv non control}
\begin{array}{l}
\mr{For~ given}~(y_0,z_0):(0,2\pi)\rightarrow \mb{R}^2,
~u:(0,2\pi)\times(0,T)\rightarrow \mb{R},\\
\mr{Find}~(y,z):(0,2\pi)\times(0,T)\rightarrow \mb{R}^2\mr{~such~ that}\\
  \left\{\begin{array}{ll}
   \partial_t y=\Delta y+\alpha z+\mathds{1}_{\omega}u&\mr{~in~}(0,2\pi)\times(0,T),\\
    \partial_t z=\Delta z&\mr{~in~}(0,2\pi)\times(0,T),\\
        y(0,\cdot)=y(2\pi,\cdot)=z(0,\cdot)=z(2\pi,\cdot)=0&\mr{~on}~(0,T),\\
	y(\cdot,0)=y_0,~z(\cdot,0)=z_0 &~\mr{in}~(0,2\pi)
        \end{array}
\right.
\end{array}
\end{equation}
 is $\Pi_1$-null controllable, 
more particularly for the initial conditions 
$y(0)=\overline{p}_0$ and $z(0)=\overline{q}_0$ in $(0,2\pi)$, there exists a control 
$u$ in $L^2((0,2\pi)\times(0,T))$ such that the solution $(y,z)$ to System 
(\ref{partiel:syst prim preuv non control}) satisfies $y(T)\equiv0$ in $(0,2\pi)$. 
We remark now that $(p,q):=(y|_{(0,\pi)},z|_{(0,\pi)})$ is a solution of (\ref{partiel:syst prim bord preuve int}) 
with $(p(0),q(0))=(p_0,q_0)$ in $(0,\pi)$, $v(t)=y(\pi,t)$  in $(0,T)$ 
and satisfying $p(T)\equiv0$ in $(0,\pi)$. This contradicts that for any control $v\in L^2(0,T)$ 
the solution $(p,q)$ 
to System (\ref{partiel:syst prim bord preuve int}) can  not be identically equal to zero at time T.
\end{proof}

\subsection{Numerical illustration}\label{partiel:section num}

\hspace*{4mm} In this section, we illustrate numerically the results obtained previously 
in Sections \ref{partiel:example control} and \ref{partiel:section contre exemple}. 
%
%
%
%
%
We adapt the HUM method to our control problem. 
For all penalty parameter  $\varepsilon>0$, we  compute the control that minimizes 
the penalized HUM functional $F_{\varepsilon}$  given by
\begin{equation*}
F_{\varepsilon}(u):=\frac{1}{2}\|u\|_{L^2(\omega\times(0,T))}^2
+\dfrac{1}{2\varepsilon}\|y(T;y_0,u)\|_{L^2(\Omega)}^2,
\end{equation*}
where $y$ is the solution to (\ref{partiel:syst prim}). 
We can find in \cite{BoyerHumEuler} the argument relating the null/approximate controllability 
and this kind of functional. 
Using the Fenchel-Rockafellar theory 
(see \cite{ekelandtemam74}  p. 59) 
we know that the minimum of $F_\varepsilon$ is equal to the opposite 
of the minimum of $J_{\varepsilon}$, the so-called dual functional,  
defined for   all $\varphi_0\in L^2(\Omega)$ by 
\begin{equation*}\begin{array}{rcl}
J_{\varepsilon}(\varphi_0)
&:=&\frac{1}{2}\|\varphi\|_{L^2(q_T)}^2
+\frac{\varepsilon}{2}\|\varphi_0\|_{L^2(Q_T)}^2
+\langle y(T;y_0,0), \varphi_0\rangle_{L^2(\Omega)},
\end{array}\end{equation*}
where  $\varphi$ is the solution to the backward  System  (\ref{partiel:def Lambda}).
Moreover  the minimizers $u_{\varepsilon}$ 
and $\varphi_{0,\varepsilon}$ of the functionals $F_{\varepsilon}$ and $J_{\varepsilon}$ respectively, 
are related through the equality 
$ u_{\varepsilon}=\mathds{1}_{\omega}\varphi_\varepsilon$, 
where $\varphi_\varepsilon$ is the solution to the backward  System (\ref{partiel:def Lambda}) 
with the initial data 
$\varphi(T)=\varphi_{0,\varepsilon}$. 
A simple computation leads to
 \begin{equation*}
  \nabla J_{\varepsilon}(\varphi_0)=\Lambda \varphi_0+\varepsilon\varphi_0+y(T;y_0,0), 
 \end{equation*}
 with the Gramiam operator $\Lambda$ defined as follows 
\begin{equation*}\begin{array}{rccl}        
 \Lambda :&L^2(\Omega)&\mapsto &L^2(\Omega),\\
 &\varphi_0&\rightarrow &w(T),
                 \end{array}
\end{equation*}
where $w$ is the solution to the following backward and forward systems
\begin{equation}\label{partiel:def Lambda}
\left\{
\begin{array}{ll}
-\partial_t \varphi=\Delta\varphi  &~\mr{in}~Q_T,\\
\varphi = 0 &~\mr{on}~\Sigma_T,\\
\varphi (T) = \varphi_0& ~ \mr{in}~ \Omega
\end{array}\right.
\end{equation}
and 
\begin{equation}\label{partiel:def Lambda2}
\left\{ \begin{array}{ll}
\partial_t w =\Delta w +\mathds{1}_{\omega}\varphi&~ \mr{in}~ Q_T, \\ 
 w=0,&~\mr{on}~ \Sigma_T,\\ 
 w(0)=0& ~ \mr{in}~ \Omega.
 \end{array}\right.
\end{equation}
Then the minimizer $u_{\varepsilon}$  of $F_{\varepsilon}$ will be computed with the help of 
the minimizer $\varphi_{0,\varepsilon}$ of $J_{\varepsilon}$ which is the solution  to the linear problem
\begin{equation*}
 (\Lambda+\varepsilon)\varphi_{0,\varepsilon}=-y(T;y_0,0).
\end{equation*}


\begin{rem}\label{partiel:rem num}
The proof of Theorem 1.7 in \cite{BoyerHumEuler} can be adapted to prove that 
\begin{enumerate}[(i)]
 \item System (\ref{partiel:syst prim}) is $\Pi_1$-null controllable if and only if 
 $\sup\limits_{\varepsilon>0}\left(\inf\limits_{L^2(\omega\times(0,T))}F_{\varepsilon}\right)<\infty,$
 \item System (\ref{partiel:syst prim}) is $\Pi_1$-approximately controllable if and only if 
 $y_{\varepsilon}(T)\underset{\varepsilon\rightarrow0}{\longrightarrow}0$,
\end{enumerate}
where $y_{\varepsilon}$ is the solution to System (\ref{partiel:syst prim}) for the control $u_{\varepsilon}$.
\end{rem}


System (\ref{partiel:syst prim}) 
with $T=0.005$, $\Omega:=(0,2\pi)$, $\omega:=(0,\pi)$ and  
$y_0:=100\sin(x)$ has been considered. 
We take the two expressions below for the coupling coefficient $\alpha$ that 
correspond respectively to Cases (1)-(2) and (3) in Theorem \ref{partiel:theo example contr intro}:
\begin{enumerate}[(a)]
 \item $\alpha(x)=1$,
 \item  $\alpha(x)=\sum\limits_{p\geqslant0}\frac{1}{p^2}cos(15px)$.
\end{enumerate}
Systems (\ref{partiel:syst prim}) and \eqref{partiel:def Lambda}-\eqref{partiel:def Lambda2}  are discretized 
with backward Euler time-marching scheme 
(time step $\delta t=1/400$) and standard 
piecewise linear Lagrange finite elements 
on a uniform mesh of size $h$ successively equal to $2\pi/50$, $2\pi/100$, $2\pi/200$ and $2\pi/300$. 
We follow the methodology of F. Boyer (see \cite{BoyerHumEuler}) 
that introduces a penalty parameter 
$\varepsilon=\phi(h):=h^4$. 
We denote by $E_h$, $U_h$ and $L_{\delta t}^2(0,T;U_h)$ the fully-discretized spaces associated to 
$L^2(\Omega)$, $L^2(\omega)$ and $L^2(q_T)$. $F^{h,\delta t}_{\varepsilon}$ is the 
discretization of $F_{\varepsilon}$ and $(y_{\varepsilon}^{h,\delta t},z_{\varepsilon}^{h,\delta t}, 
u_{\varepsilon}^{h,\delta t})$ the solution to the corresponding fully-discrete problem of minimisation. 
For more details on  the fully-discretization of System (\ref{partiel:syst prim}) 
and  Gramiam  $\Lambda$ (used to the minimisation of $F_{\epsilon}$),  
we refer to Section 3 in \cite{BoyerHumEuler} and in \cite[p. 37]{glo95}  respectively. 
The results are depicted Figure \ref{partiel:courbe ordre constant} and \ref{partiel:courbe ordre sumcos}. 

\medskip


\begin{center}
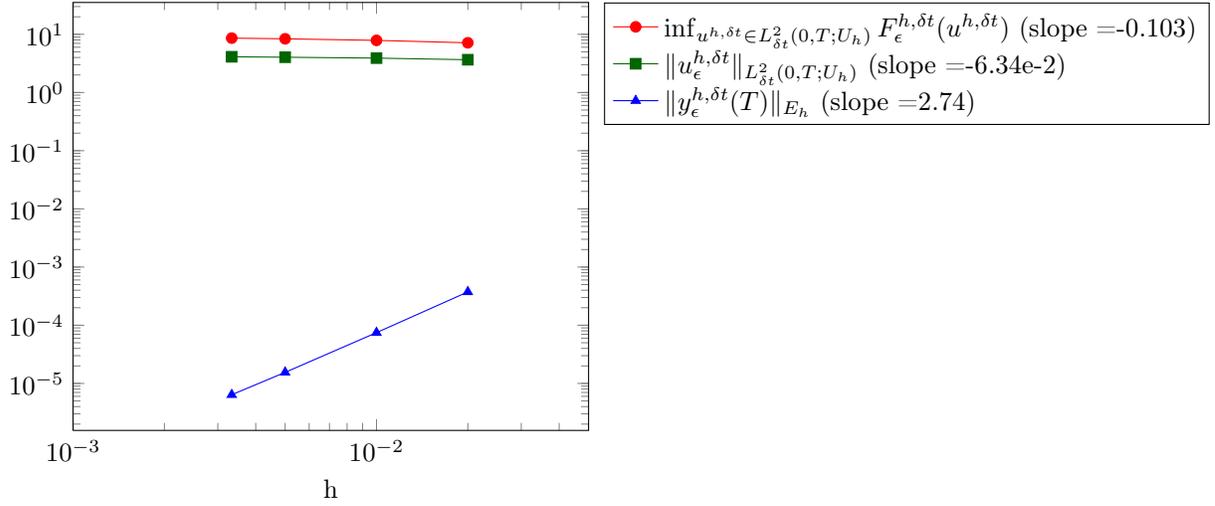
\begin{figure}[H]
\begin{tikzpicture} \begin{loglogaxis}[xlabel=h,xmin=1e-3,xmax=5e-2
,legend cell align=left,legend pos=outer north east, legend columns=1]
   \addplot[color=red,mark=*]coordinates { 
(1/50,7.15) 
(1/100,7.87)
(1/200, 8.37)
(1/300,8.60) 
 };
 \addplot[color= black!60!green,mark=square*]  coordinates { 
(1/50,3.66) 
(1/100,3.90)
(1/200, 4.04)
(1/300,4.11) 
 }; 
\addplot[color=blue,mark=triangle*] coordinates { 
(1/50,3.74e-4)
(1/100,7.46e-5)
(1/200,1.55e-5)
(1/300,6.39e-6) 
 };
 \legend{ $\inf_{u^{h,\delta t}\in  L^2_{\delta t}(0,T;U_h)} F_{\epsilon}^{h,\delta t}(u^{h,\delta t})$ (slope =-0.103), 
 $\|u_{\epsilon}^{h,\delta t}\|_{L^2_{\delta t}(0,T;U_h)}$ (slope =-6.34e-2), 
$\|y_{\epsilon}^{h,\delta t}(T)\|_{E_h} $ (slope =2.74) }
\end{loglogaxis} \end{tikzpicture}\caption{\label{partiel:courbe ordre constant}
Minimal value of the functional $\inf_{u^{h,\delta t}\in  L^2_{\delta t}(0,T;U_h)} F_{\epsilon}^{h,\delta t}(u^{h,\delta t})$, 
norm of the control $\|u_{\epsilon}^{h,\delta t}\|_{L^2_{\delta t}(0,T;U_h)}$, 
and distance to the target $\|y_{\epsilon}^{h,\delta t}(T)\|_{E_h} $  in Case (a).}
\end{figure}
\end{center}

\begin{center}
\begin{figure}[H]
\begin{tikzpicture} \begin{loglogaxis}[xlabel=h,xmin=1e-3,xmax=5e-2
,legend cell align=left,legend pos=outer north east, legend columns=1]
   \addplot[color=red,mark=*]coordinates { 
(1/50,0.01) 
(1/100,0.13)
(1/200, 1.93)
(1/300,8.92) 
 };
 \addplot[color= black!60!green,mark=square*]  coordinates { 
(1/50,0.0508) 
(1/100,0.12)
(1/200, 0.42)
(1/300,1.1) 
 }; 
\addplot[color=blue,mark=triangle*] coordinates { 
(1/50,5.3e-5)
(1/100,5.03e-5)
(1/200,4.8e-5)
(1/300,4.53e-5) 
 };
 \legend{ $\inf_{u^{h,\delta t}\in  L^2_{\delta t}(0,T;U_h)} F_{\epsilon}^{h,\delta t}(u^{h,\delta t})$ (slope =-3.80), 
 $\|u_{\epsilon}^{h,\delta t}\|_{L^2_{\delta t}(0,T;U_h)}$ (slope =-1.70), 
$\|y_{\epsilon}^{h,\delta t}(T)\|_{E_h} $ (slope =8.34e-2) }
\end{loglogaxis} \end{tikzpicture}\caption{\label{partiel:courbe ordre sumcos}
Minimal value of the functional $\inf_{u^{h,\delta t}\in  L^2_{\delta t}(0,T;U_h)} F_{\epsilon}^{h,\delta t}(u^{h,\delta t})$, 
norm of the control $\|u_{\epsilon}^{h,\delta t}\|_{L^2_{\delta t}(0,T;U_h)}$, 
and distance to the target $\|y_{\epsilon}^{h,\delta t}(T)\|_{E_h} $ in Case (b).}
\end{figure}
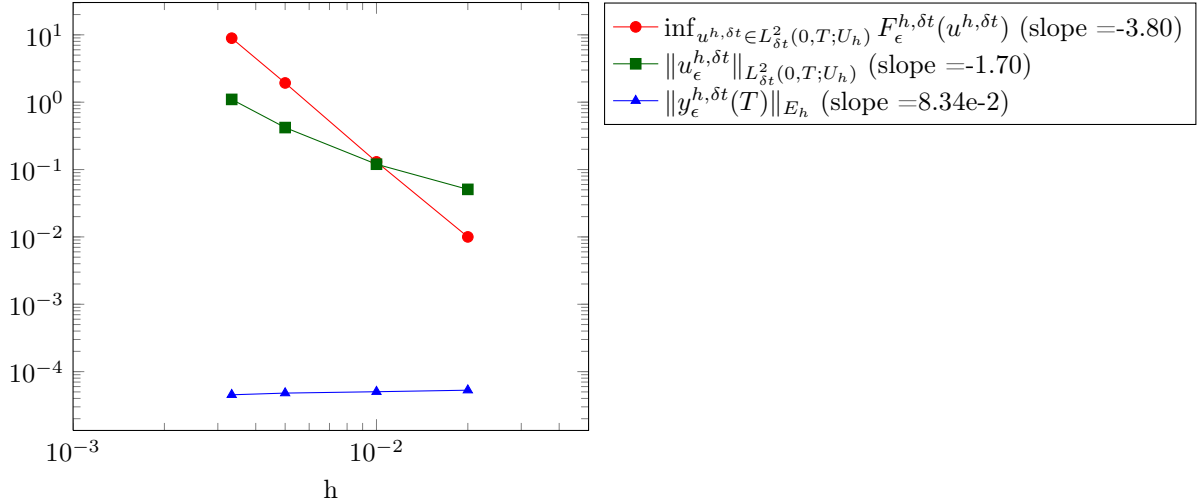
\end{center}

%



As mentioned in the introduction of the present article (see Theorem \ref{partiel:theo example contr intro}), 
in both situations (a) and (b), System (\ref{partiel:syst prim}) is 
$\Pi_1$-approximately controllable and we observe indeed in Figure \ref{partiel:courbe ordre constant} and \ref{partiel:courbe ordre sumcos} that 
the norm of the numerical solution to System (\ref{partiel:syst prim}) 
at time $T$ ($\textcolor{blue}{-\blacktriangle-}$) is decreasing 
when  reducing the penality parameter $\varepsilon=h^4$.

In Figure \ref{partiel:courbe ordre constant}, the minimal value of the 
functional $F^{h,\delta t}_{\varepsilon}$ ($\textcolor{red}{-\bullet-}$) 
as well as the $L^2$-norm of the control $u^{h,\delta t}_{\varepsilon}$ 
($\textcolor{black!60!green}{-\blacksquare-}$) 
remain roughly constant whatever is the value of $h$ (and $\varepsilon=h^4$).
This appears in agreement with the results (1)-(2) of Theorem \ref{partiel:theo example contr intro}, that state 
the $\Pi_1$-null controllability of System (\ref{partiel:syst prim}) 
in Case (a) of a constant coupling coefficient $\alpha$ (see Remark \ref{partiel:rem num} (i)). 
Furthermore 
the convergence to the null target is approximately of order $2$ (slope of $2.27$). 
This is in agreement with the convergence rate established in
\cite[Proposition 2.2]{BoyerHumEuler},
which should be $h^2$ for $\varepsilon = h^4$ (this result should be in fact slightly 
adapted to consider $\Pi_1$-null controllability).

At the opposite, in Figure \ref{partiel:courbe ordre sumcos}, the 
minimal value of the functional $F^{h,\delta t}_{\varepsilon}$ as well as the 
$L^2$-norm of the control $u^{h,\delta t}_{\varepsilon}$
 are strongly increasing whenever $h$ (and $\varepsilon$) become smaller.
This coincides with point (3) of Theorem \ref{partiel:theo example contr intro}: 
for the chosen value of the coupling coefficient $\alpha$ in Case (b), no
$\Pi_1$-null controllability of System (\ref{partiel:syst prim}) is expected.
Moreover, convergence to the null target is quite slow, with a slope 
of approximately $8.34e-2$.

\quad \textit{Acknowledgements}. The authors thank Assia Benabdallah and Franck Boyer
  for their interesting comments and suggestions. They thank as well  the two  referees for their remarks that helped to improve the paper.

\bibliographystyle{AIMS}
\bibliography{biblio.bib}

\end{document}